\documentclass[a4paper, 12pt]{article}
\usepackage[english]{babel}
\usepackage{amsmath, amsthm, amssymb} 
\usepackage{comment} 
\usepackage{tikz-cd} 
\usepackage{geometry} 
\usepackage[normalem]{ulem} 
\usepackage[backend=biber, maxnames=99, style=alphabetic]{biblatex} 
\usepackage{csquotes} 
\usepackage{xcolor} 
\usepackage[export]{adjustbox} 
\usepackage{pdfpages}
\usepackage{hyperref} 
\hypersetup{
    colorlinks,
    citecolor=black,
    filecolor=black,
    linkcolor=black,
    urlcolor=black
}

\makeatletter 
\newcommand*\bigcdot{\mathpalette\bigcdot@{.8}}
\newcommand*\bigcdot@[2]{\mathbin{\vcenter{\hbox{\scalebox{#2}{$\m@th#1\bullet$}}}}}
\makeatother

\makeatletter 
\newcommand{\flim@}[2]{
    \vtop{\m@th\ialign{##\cr
    \hfil$#1\operator@font lim$\hfil\cr
    \noalign{\nointerlineskip\kern1.5\ex@}#2\cr
    \noalign{\nointerlineskip\kern-\ex@}\cr}}%
}
\newcommand{\flim}{%
    \mathop{\mathpalette\varlim@{\leftarrowfill@\scriptscriptstyle}}\nmlimits@
}
\makeatother

\theoremstyle{plain}
\newtheorem{theorem}{Theorem}[section]
\newtheorem{corollary}[theorem]{Corollary}
\newtheorem{lemma}[theorem]{Lemma}

\newtheorem{prop}[theorem]{Proposition}
\newtheorem*{theorem_introduction}{Kadeishvili's Theorem}

\theoremstyle{definition}
\newtheorem{definition}[theorem]{Definition}

\theoremstyle{remark}
\newtheorem{remark}[theorem]{Remark}

\newgeometry{vmargin={30mm}, hmargin={30mm}} 

\title{Formality of $A_\infty$-Algebras}
\author{Carl Felix Waller\footnote{wallerfelix@gmail.com}}
\date{}

\usepackage{titlesec} 
\titleformat{\section}[block]{\bfseries\filcenter}{\thesection}{1em}{}
\titleformat{\subsection}[hang]{\bfseries}{\thesubsection}{1em}{}

\begin{document}

\maketitle
\abstract{This master's thesis contains an introduction to $A_\infty$-algebras and homological perturbation theory. We then discuss the formality of compact K\"ahler manifolds and present a direct proof of a homotopy transfer principle of $A_\infty$-algebras, also known as Kadeishvili's Theorem.}
\section{Introduction}
Given a graded vector space $A$ equipped with a multiplication $m_2$ and a square zero derivation $m_1$, we say that $A$ is an $A_\infty$\textit{-algebra} if there exists a possibly infinite sequence of higher homotopies $m_n:A^{\otimes n}\rightarrow A$ for $n\geq 3$, which measure to what extent $m_2$ fails to be associative. Explicitly, when we sum over all possible combinations of maps $A^{\otimes n}\rightarrow A$ the \textit{higher order associatity conditions} must be satisfied:
\begin{equation} \label{eq:a_infty_introduction}
\sum_{\bigcdot} \pm m_{\bigcdot}(1^{\bigcdot}\otimes m_{\bigcdot}\otimes 1^{\bigcdot})=0.
\end{equation}
For $n=1$, equation \eqref{eq:a_infty_introduction} restates the fact that $m_1$ squares to zero. For $n=2$, equation \eqref{eq:a_infty_introduction} says that $m_1$ is a derivation with respect to $m_2$. The first new equation is $n=3$, which tells us that the associator of $m_2$ is equal to the boundary of $m_3$. In this sense, $A_\infty$-algebras can be seen as a generalization of differential graded algebras, where the associativity condition is relaxed. $A_\infty$-algebras have better deformation properties than differential graded algebras. However, seemingly simple tasks, such as checking that composition of $A_\infty$-algebra morphisms are well-defined, are already non-trivial. The main purpose of this thesis is to give an introductory account of $A_\infty$-algebras, and carry out computations at an extreme level of detail not usually found in the literature.

\vspace{4mm}

We begin with a short discussion of the history of $A_\infty$-algebras. The topological origin of $A_\infty$-algebras goes back to the PhD thesis of James Stasheff \cite{stasheff_H_space_1}, \cite{stasheff_H_space_2}. In it, he uses $A_\infty$-spaces to characterize loop spaces. The story of recognizing loop spaces began with Milnor, who proved that many topological groups have the homotopy type of loop spaces \cite[Theorem 5.2 (1)]{milnor}. This result was sharpened by Dold and Lashof, who proved that more generally, many associative $H$-spaces have the homotopy type of loop spaces \cite[6.2 Theorem]{lashof_dold}. An obvious question is whether the assumption of an associative $H$-space can further be weakened so that one gets an ``if and only if'' statement. The way was paved by Sugawara \cite{sugawarah_space}, \cite{sugawara_group}, who showed that the key idea is to relax the associativity assumption of the multiplication, i.e.\ allow the multiplication to be merely homotopy associative. Stasheff attempted to further relax Sugawara's assumption of the existence of a homotopy inverse and found that he had to introduce higher homotopies. He was able to improve upon Sugawara's results and show that under mild assumptions a space $X$ has the homotopy type of a loop space if and only if $X$ is an $A_\infty$-space, i.e.\ $X$ admits a unital multiplication $M_2$ and higher homotopies 
\[
M_n:K_n \times X^n \rightarrow X,\quad \forall n\geq 3
\]
which measure the failure of $M_2$ to be associative. Here $K_n$ are the Stasheff polytopes which are convex polytopes similar to simplices but with a boundary that is built inductively out of copies of $K_1,...,K_{n-1}$. For more detail on Stasheff's proof see \cite[§6-2]{kane}. A summary of research on $H$-spaces and $A_\infty$-spaces is given, for example, in the lecture notes \cite{stasheff_book}.

\vspace{4mm}

$A_\infty$-spaces continued to play a role in homotopy theory: We mention here Boardman and Vogt (\cite{boardman_hom_everthing}, \cite{boardman_book}), Adams \cite{adams} and finally Peter May \cite{may}, who revolutionized the field by characterizing all iterated loop spaces with his introduction of the concept of an operad. In the languages of operads, there exists a non-symmetric operad $\mathcal{A}_\infty$, and algebras over $\mathcal{A}_\infty$ are what we call $A_\infty$-algebras (see \cite{markl_simplex} or \cite{markl_operads_book} for more detail).

\vspace{4mm}

The notion of an $A_\infty$-space naturally leads to the simpler notion of an $A_\infty$-algebra: Given a CW-complex $X$ with an $A_\infty$-space structure, one can show that the cellular chain complex of $X$ has the structure of an $A_\infty$-algebra \cite[Theorem 2.3]{stasheff_H_space_2}. For example, if $\Omega(X)$ is any loop space, by Milnor \cite{Milnor_cw_spaces} there exists a CW-complex $Y$ that has the same homotopy type as $\Omega(X)$. The concatenation map carried over to $Y$ can be deformed to a strictly unital multiplication and higher homotopies in the loop space can also be deformed in $Y$ so that $Y$ carries the structure of an $A_\infty$-space. The space of cellular chains on $Y$ may serve as the prime example of an $A_\infty$-algebra. After their discovery, $A_\infty$-algebras continued to find applications in topology (\cite{smirnov}, \cite{proute}, \cite[Theorem 6.4]{huebschmann_homotopy_type}, \cite{huebschmann_metacyclic}, and of particular importance for this thesis \cite{kadeishvili_russian}). Later $A_\infty$-algebras also started to appear in geometry and physics (\cite{getzler}, \cite{stasheff_book_quasi_hopf}, \cite{fukaya_morse}, \cite{kontsevich_mirror_symmetry},
\cite{kontsevich_soibelman}, \cite{seidel_fukaya}, \cite{FOOO})

\vspace{4mm}

We now give a brief overview of the content of this thesis: In Section \ref{sec:notation}, we set up some notation used throughout the thesis, and in Section \ref{sec:a_infty}, we give an introduction to $A_\infty$-algebras. In Section \ref{sec:formality}, we start by defining minimal models of $A_\infty$-algebras. Minimal models of differential graded commutative algebras were first introduced by Sullivan \cite{sullivan_minimal_model}, following early developments by Quillen \cite{quillen}. For a textbook introduction to the theory of minimal models, see \cite{yves_rational_homotopy}. Minimal models are an important tool in rational homotopy theory \cite{griffiths_rational_homotopy}. A particularly interesting class of spaces in rational homotopy theory are formal spaces, i.e.\ spaces that admit a minimal model which is formal. More generally, we say that any differential graded commutative algebra $A$ is formal if we can connect it to its homology via quasi-isomorphisms:
\[
(H_*(A),0)\xleftarrow{\simeq} \bigcdot \ ... \ \bigcdot \xrightarrow{\simeq} (A,d).
\]
In Section \ref{sec:formality}, we proceed to define a notion of formality for balanced $A_\infty$-algebras. Balanced $A_\infty$-algebras (also called commutative $A_\infty$-algebras or $C_\infty$-algebras) are $A_\infty$-algebras whose higher homotopies vanish on shuffle products (\cite{kadeishvili_C_infty}, \cite{markl_C_infty}). At the end of Section \ref{sec:formality}, the formality of compact K\"ahler manifolds in terms of the $A_\infty$-algebra definition of formality is proven. The rest of the thesis focuses more on minimal models in general and not specifically their formality.

\vspace{4mm}

In Section \ref{sec:quasi-cofree}, we discuss the one-to-one correspondence between $A_\infty$-algebras and certain differential graded coalgebras. It was already Stasheff who noticed this equivalent description of $A_\infty$-algebras in his original paper \cite{stasheff_H_space_2}. This is the basis for the use of homological perturbation theory to study $A_\infty$-structures. The development of homological perturbation theory predates the discovery of $A_\infty$-algebras: \cite{eilenberg_mac_lane}, \cite{brown}. A fundamental result regarding minimal models of $A_\infty$-algebras is the following theorem:

\begin{theorem_introduction} (Paraphrasing Theorem \ref{thm:minimal_model})
Let $(A,m)$ be an $A_\infty$-algebra and $H_*(A)$ the homology of the chain complex $(A,m_1)$. There exists an $A_\infty$-structure ($H_*(A),m')$ and a morphism of $A_\infty$-algebras
\[
i: (H_*(A),m') \rightarrow (A,m)
\]
such that $m'_1=0$, $m_2'$ is the induced product and $i_1$ induces the identity on homology. Moreover, the $A_\infty$-structure on $H_*(A)$ is unique up to isomorphism of $A_\infty$-algebras.
\end{theorem_introduction}

This theorem was proved by Kadeishvili \cite{kadeishvili_russian} for the case where $A$ is a differential graded algebra\footnote{Kadeishvili's argument may also be extended to the general case \cite{zhou} or \cite{petersen}.}. For an english version of Kadeishvili's paper see \cite{kadeishvili}. Kadeishvili and Huebschmann \cite[$2.1_*$]{HUKA} and Gugenheim, Lambe and Stasheff (\cite{gugenheim_89}, \cite[Theorem 4.2]{gugenheim_91}) almost contemporaneously published a new way to prove this result in full generality. The idea is to prove everything in the coalgebra setting via the perturbation lemma. The perturbation lemma allows perturbations of differentials to be carried along certain chain equivalences. In fact, once the notation and theory are set up the perturbation lemma allows for a very streamlined proof of Kadeishvili's theorem. We illustrate this in Section \ref{sec:pertub_lemma}. There are many generalizations of Kadeishvili's Theorem with similar methods of proof (\cite{manetti}, \cite{berglund}).

\vspace{4mm}

Another ansatz is to show that the cobar construction of certain cooperads are cofibrant. In this case the result follows from more general but also more abstract statements about operads (\cite{Moerdijk}, \cite{Hinich}, \cite{markl_homotopy_algebras}, \cite{yau}). Markl also gives explicit formulas in \cite{markl_a_infty} derived from the abstract statements in \cite{markl_homotopy_algebras}.

\vspace{4mm}

A different way to deal with $A_\infty$-structures is by representing the $A_\infty$-relations by sums over planar trees. This idea was first proposed by Konsevith and Soibelman \cite{kontsevich_soibelman} and was further developed in \cite{Laan} and by Loday and Vallette \cite[Theorem 10.3.3]{loday_vallette}. More recently, Hicks (\cite[Theorem 2.2.1]{hicks_algebras}, \cite[Appendix A]{hicks}) shows that the method of using trees can be extended to curved $A_\infty$-algebras.

\vspace{4mm}

Merkulov \cite{merkulov} showed that there is a way to prove the existence part of Kadeishvili's Theorem via a straightforward induction argument. In Section \ref{sec:explicit_proofs} we establish that the formulas for minimal models as obtained in Section \ref{sec:pertub_lemma} coincide with the formulas in \cite{merkulov} up to sign difference. Following this, we give a direct induction proof of the existence part of Kadeishvili's Theorem. The main part of the proof is concerned with reproving that the projection map obtained from the perturbation lemma is, in fact, an $A_\infty$-morphism. The main difficulty is correctly reordering the large nested sums. We do not claim that these methods are unknown to Merkulov or other experts in the field. However, as far as we know, these computations appear nowhere else in the literature. 

\vspace{4mm}

Lastly, in Appendix \ref{sec:appendix_b} we show that the same sum manipulation methods can be used to give an explicit proof of the fact that the composition of $A_\infty$-morphisms is well-defined.

\section*{}
\textbf{Acknowledgement.} The author is indebted and very grateful to his supervisor Will J. Merry for introducing him to the subject of $A_\infty$-algebras. His constant support and valuable comments during numerous conversations have truly been a great source of motivation and inspiration.

\section{Notation} \label{sec:notation}
Let $k$ be an arbitrary field and $(C,\Delta,\epsilon)$ a coalgebra over $k$. Define the \textit{iterated coproduct} by $\Delta_1 := \Delta$ and inductively $\Delta_{n+1}:= (\Delta\otimes \mathrm{id}^{\otimes n})\circ \Delta_n$. For any $c\in C$ the iterated coproduct in Sweedler notation is denotes as:
\[
\Delta_{n-1}(c) =\sum_{(c)} c_{(1)} \otimes ... \otimes c_{(n)}.
\]
We refer to Sweedler's book \cite{coalgebras} for a detailed introduction to coalgebras. If $u:k\rightarrow C$ is a coaugmentation map, i.e. there is a choice of unit $1\in C$, we set $\bar{C}:= \ker \epsilon$ and define the reduced coproduct as:
\begin{align*}
	\bar{\Delta}: &\bar{C}\rightarrow \bar{C}\otimes \bar{C}\\
	&c\mapsto \Delta(c) - c\otimes 1 - 1\otimes c.
\end{align*}
$(\bar{C},\bar{\Delta})$ is a ``coalgebra without a counit'', which we refer to as a \textit{non-counitial coalgebra}. If $(C,\Delta,\epsilon,u)$ is a coaugmented coalgebra we call the associated non-counital coalgebra $(\bar{C},\bar{\Delta})$, the \textit{reduced coalgebra of C}. We define the reduced iterated product as $\bar{\Delta}^1 := \bar{\Delta}$ and inductively $\bar{\Delta}^{n+1} := ( \bar{\Delta}\otimes \mathrm{id}^{\otimes n})\circ \bar{\Delta}^{n}$. In Sweedler notation we write for $c\in \bar{C}$ \[
\bar{\Delta}^n(c) =: \sum_{(c)} c^{(1)}\otimes ... \otimes c^{(n)}.
\]
We use superscripts to distinguish this from the other iterated product.

\vspace{4mm}

Next we introduce some basic conventions and definitions of graded objects, closely following \cite{loday_vallette}: Let $V$ be a vector space over $k$. We say that $V$ is a \textit{graded vector space} if there exists a family of vector spaces $\{V_n\}_{n\in\mathbb{Z}}$ such that
\[
V = \bigoplus_{n\in \mathbb{Z}}V_n.
\]
An element $v\in V_n$ is called \textit{homogeneous} and of \textit{degree} $n$, we write $|v| = n$.

\begin{definition}
    Let $V = \bigoplus_n V_n$ and $W = \bigoplus_n W_n$ be two graded vector spaces. A linear map $f:V\rightarrow W$ is called a \textit{morphism of graded vector spaces of degree r}, if there are linear maps $f_n :V_n \rightarrow W_{n+r}$ such that $f = \bigoplus_n f_n$.
\end{definition}
If the degree of a morphism of graded vector spaces is not specified, it is always assumed to be $0$. The tensor product of two graded vector spaces is defined as the graded vector space $V\otimes W = \bigoplus_n (V\otimes W)_n$ with
\[
(V\otimes W)_n := \bigoplus_{n=i+j} V_i\otimes W_j.
\]
Let $s_{\bigcdot} k$ denote the graded vector space of dimension $1$ generated by one element $s$ in degree $1$. Given a graded vector space $V$, we call
\[
sV := s_{\bigcdot} k \otimes V
\]
the \textit{suspension} of $V$. The obvious isomorphism $s:V\rightarrow sV$ is called the \textit{suspension map}. The tensor product makes the category of graded vector spaces into a monodical category. For $V, W$ two graded vector spaces we define a map
\[
\tau_{V,W} : V\otimes W \rightarrow W\otimes V
\]
which maps homogeneous elements as follows
\[
v\otimes w \mapsto (-1)^{|v||w|}w\otimes v.
\]
$\tau$ is called the switching map, and makes the monodical category into a symmetric monodical category, also called the category of \textit{sign-graded vector spaces}. When defining the tensor product of maps, we use the Koszul sign convention:
\begin{lemma}
	Let $f:V\rightarrow V'$ and $g:W\rightarrow W'$ be two morphisms of graded vector spaces of degree $|f|, |g|$, respectively. Then there is a unique morphism of graded vector spaces $ f\otimes g: V\otimes W \rightarrow V'\otimes W' $ of degree $|f|+ |g|$ such that for $w\in W$ and homogeneous elements $v\in V$
	\[
	f\otimes g (v\otimes w) = (-1)^{|v||g|} (f(v)\otimes g(w)).
	\]
	Moreover, if $f':V'\rightarrow V''$ and $g':W'\rightarrow W''$ are two more morphisms of graded vector spaces then 
	\[
	(f\otimes g) \circ (f'\otimes g') = (-1)^{|g| |f'|} (f'\circ f)\otimes (g'\circ g).
	\]
\end{lemma}

Let us recall some basic definitions for algebras and the dual definitions for coalgebras:
\begin{definition}
	A \textit{graded algebra} is a graded vector space $A$ and an algebra, such that the product and the unit are maps of degree $0$. A linear map $d:A\rightarrow A$ of degree $-1$ is called a \textit{derivation} if
	\[
	 d \circ \nu = \nu \circ (d\otimes \mathrm{id} + \mathrm{id}\otimes d).
	\]
	A \textit{differential graded algebra} is a graded algebra equipped with a derivation that squares to zero.
\end{definition}
It immediately follows that a derivation satisfies $d(1) =0$. We abbreviate differential graded algebra by \textit{dg-algebra}. There are analogue definitions in the dual setting:
\begin{definition} \label{def:dg-coalgebra}
	A \textit{graded coalgebra} is a graded vector space $C$ and a coalgebra, such that the coproduct and counit are maps of degree $0$. A linear map $d:C\rightarrow C$ of degree $-1$ is called a \textit{coderivation} if
	\begin{equation} \label{eq:coderivation}
		\Delta \circ d = (d\otimes \mathrm{id} + \mathrm{id}\otimes d) \circ \Delta.
	\end{equation}
	A \textit{differential graded coalgebra} is a graded coalgebra equipped a coderivation that squares to zero.
\end{definition}
Denote by $\mathrm{Coder}(C)$ the $k$-vector space of coderivations. Note that $d(1) = 0$ does \uline{not} follow for coderivations on coalgebras. We abbreviate differential graded coalgebra by \textit{dg-coalgebra}. A differential graded coalgebra is \textit{coaugmented} if the coaugmentation map is of degree $0$; here the graded coalgebra $k$ has all elements in degree $0$.

\section{\texorpdfstring{$A_\infty$}{TEXT}-Algebras}
\label{sec:a_infty}
This section introduces the notion of an $A_\infty$-algebra and relates this to the notion of a dg-algebra. For a basic introduction to $A_\infty$-algebras also see \cite{keller_long}, \cite{keller_short}, \cite{keller_addendum} or \cite{markl_operads_book}. For a more complete introduction see \cite{loday_vallette} or the PhD thesis \cite{proute} by Prout{\'e}.

\begin{definition}
	An \textit{$A_\infty$-algebra} is a graded vector space $A$ together with a family of linear maps $m_n: A^{\otimes n} \rightarrow A$ of degree $n-2$, such that for all $n\in \mathbb{N}$
	\begin{equation}\label{eq:higher_assoziativity_condition}
		\sum_{s=1}^n\sum_{\substack{r+t=n-s\\r,t\geq 0}} (-1)^{rs+t}m_{r+t+1}(\mathrm{id}^{\otimes r}\otimes m_s\otimes \mathrm{id}^{\otimes t})= 0.
	\end{equation}
	These equations are called the \textit{higher associativity conditions}. 
\end{definition}
Writing out the first three higher associativity conditions we get
\begin{align*}
	&n=1:\quad m_1 m_1 = 0,\\
	&n=2:\quad m_1 m_2 - m_2(\mathrm{id}\otimes m_1 + m_1\otimes \mathrm{id}) = 0,\\
	\mathrm{and}\ &n=3:\quad m_2 ( \mathrm{id}\otimes m_2 - m_2 \otimes \mathrm{id}) +m_1m_3 \\
	& \qquad +m_3(m_1\otimes \mathrm{id}\otimes \mathrm{id}+ \mathrm{id}\otimes m_1 \otimes \mathrm{id}+\mathrm{id}\otimes \mathrm{id}\otimes m_1) = 0.
\end{align*}
The first equation implies that $m_1$ is a differential. The second equation implies that $m_1$ behaves similarly to a derivation with respect to the $m_2$ map. The third equation implies that in general $m_2$ is not associative. We say that $m_2$ is associative up to the homotopy $m_3$.

\begin{definition}
	Let $(A,m),(A',m')$ be two $A_\infty$-algebras. A \textit{morphism of} $A_\infty$\textit{-algebras} $f:(A,m)\rightarrow (A',m')$ is a family of linear maps $f_n:A^{\otimes n}\rightarrow A'$ of degree $n-1$ such that for all $n\in \mathbb{N}$
\begin{equation}\label{eq:morphism}
\sum_{r+s+t=n} (-1)^{rs+t}f_{r+t+1}(\mathrm{id}^{\otimes r}\otimes m_s\otimes \mathrm{id}^{\otimes t})= \sum_{i_1+...+i_r = n} (-1)^{l(i_1,...,i_r)} m_r'(f_{i_1}\otimes...\otimes f_{i_r})
\end{equation}
where the sign on the right-hand side is given by
\begin{equation} \label{eq:A_infty_signs}
	l(i_1,...,i_r):= \sum_{1\leq j < k\leq n}(i_k-1)i_j.
\end{equation}
\end{definition}
If $f':(A',m')\rightarrow (A'',m'')$ is another morphism we define the composition $f'\circ f:(A,m)\rightarrow (A'',m'')$ as
\begin{equation} \label{eq:a_infty_composition}
    (f'\circ f)_n := \sum_{i_1+...+i_r = n} (-1)^{l(i_1,...,i_r)} f'_r(f_{i_1}\otimes...\otimes f_{i_r}).
\end{equation}
It is not immediate that this is well-defined. In Proposition \ref{prop:composition} we give a proof using the coalgebra describtion of $A_\infty$-algebras, and in appendix \ref{sec:appendix_b} we present a direct proof. The identity morphism $\mathrm{id}_A :(A,m)\rightarrow (A,m)$ is defined as $(\mathrm{id}_A)_1 = \mathrm{id}_A$ and $(\mathrm{id}_A)_n = 0$ for all $n>1$. This is indeed a morphism of $A_\infty$-algebras, i.e. we have for all $n\geq 1$
\begin{align*}
&\sum_{n=r+s+t} (-1)^{rs+t}(\mathrm{id}_A)_{r+t+1}(\mathrm{id}_A^{\otimes r}\otimes m_s\otimes \mathrm{id}_A^{\otimes t})\\
&= \sum_{i_1+...+i_r = n} (-1)^{l(i_1,...,i_r)}m_r((\mathrm{id}_A)_{i_1}\otimes...\otimes (\mathrm{id}_A)_{i_r}).
\end{align*}
This gives us the category of $A_\infty$-algebras. The isomorphisms take the following form:
\begin{prop} \label{prop:a_infty_isomorphism}
	A morphism of $A_\infty$-algebras $f:(A,m)\rightarrow (A',m')$ is an isomorphism if and only if $f_1$ is an isomorphism.
\end{prop}
One direction is immediate. For the other direction assume $f_1$ is an isomorphism. We inductively define the inverse $g:(A',m')\rightarrow (A,m)$, starting with $g_1 := f_1^{-1}$. Let $n>1$, we must have
\begin{equation*} 
	0= (f\circ g)_n = \sum_{n=i_1+...+i_r} (-1)^{l} f_r(g_{i_1}\otimes ... \otimes g_{i_r}),
\end{equation*}
so we set
\begin{equation} \label{eq:inverse} 
    g_n := f_1^{-1}\big(\sum_{\substack{n=i_1+...+i_r\\r>1}} (-1)^{l} f_r(g_{i_1}\otimes ... \otimes g_{i_r})\big).
\end{equation}
It remains to show that $g$ is also a left-inverse of $f$ and that $g$ is a valid morphism of $A_\infty$-algebras, this is proved in Proposition \ref{prop:inverse}. The following proposition shows that dg-algebras are $A_\infty$-algebras.

\begin{prop} \label{prop:sub_category}
	The category of dg-algebras\footnote{The category of differential graded algebras means only allowing morphisms of degree $0$} forms a sub-category of the category of $A_\infty$-algebras.
\end{prop}

\begin{proof}
Firstly for every dg-algebra $(A,d,\nu)$ we can define an $A_\infty$-structure on $A$ by setting $m_1 = d$, $m_2 = \nu$ and all higher products are set to zero. These maps have the correct degree and the higher associativity conditions are satisfied: Indeed, for $n=1$
\begin{equation}\label{eq:n=1}
	 m_1 m_1 = d^2 = 0,
\end{equation}
for $n=2$
\begin{equation}\label{eq:n=2}
	m_1 m_2 - m_2(\mathrm{id}\otimes m_1 + m_1\otimes \mathrm{id}) + = d\nu - \nu (\mathrm{id}\otimes d+ d\otimes \mathrm{id}) = 0,
\end{equation}
for $n = 3$
\begin{equation}\label{eq:n=3}
	m_2 ( \mathrm{id}\otimes m_2 - m_2 \otimes \mathrm{id}) = \nu (\mathrm{id}\otimes \nu - \nu \otimes \mathrm{id}) = 0,
\end{equation}
and for $n \geq 4$ all the summands vanish
\begin{align*}
	&\sum_{n=r+s+t}(-1)^{rs+t} m_{r+t+1} (\mathrm{id}^{\otimes r}\otimes m_s \otimes \mathrm{id}^{\otimes t})\\
	&= dm_n + \nu ( \mathrm{id}\otimes m_{n-1} - m_{n-1} \otimes \mathrm{id}) = 0
\end{align*}
as $n-1 \geq 3$. Let $(A,d,\nu)$, $(A',d',\nu')$ be two dg-algebras and denote by $(A,m)$, $(A',m')$ the corresponding $A_\infty$-algebras, respectively, as defined above. Let $f:(A,d)\rightarrow (A',d')$ be a morphism of dg-algebras. We claim that $f:(A,m) \rightarrow (A',m')$ given by $f_1 = f$ and all higher $f_n =0$ defines a morphism of $A_\infty$-algebras. So we just check equation \eqref{eq:morphism}: For $n = 1$ equation \eqref{eq:morphism} gives
\[
fd = d'f.
\]
For $n = 2$ we get
\begin{equation*}
	f \nu = \nu' (f\otimes f).
\end{equation*}
Let $n\geq 3$, using that $m_s = 0$ for all $s\geq 3$, the left-hand side of equation \eqref{eq:morphism} becomes
\[
\sum_{r+t=n-1}(-1)^{n-1} f_n (\mathrm{id}^{\otimes r}\otimes d\otimes \mathrm{id}^{\otimes t}) + \sum_{r+t=n-2} (-1)^t f_{n-1}(\mathrm{id}^{\otimes r}\otimes \nu \otimes \mathrm{id}^{\otimes t}).
\]
However, this is zero as $f_n=0$ and $f_{n-1} =0$, because $n,n-1 \geq 2$. Similarly the right-hand side of equation \eqref{eq:morphism} is equal to
\[
\sum_{i_1 +i_2= n}\pm \nu'(f_{i_1}\otimes f_{i_2}) + d' f_n = \nu'(f_{1}\otimes f_{n-1} + (-1)^{n-1}f_{n-1}\otimes f_1) = 0.
\]
Lastly, composition of two such maps in the category of $A_\infty$-algebras, is clearly just regular composition of morphisms of dg-algebras. This completes the proof.
\end{proof}
It is worth mentioning that the category of dg-algebras is \uline{not} a full subcategory of the category of $A_\infty$-algebras. A counter example is the minimal model map of a compact K\"ahler manifold (see the end of Section \ref{sec:formality}).

\section{Formality}
\label{sec:formality}
Let us define dg-algebras which are commutative up to sign:
\begin{definition}
    A dg-algebra $(A,d,\nu)$ is called \textit{graded-commutative} if
    \[
    \nu\circ \tau = \nu.
    \]
\end{definition}
Here $\tau$ is the switching map from Section \ref{sec:notation}. We abbreviate graded-commutative dg-algebra with \textit{dgc-algebra}. Recall the classical definition of formality:
\begin{definition}
    A dgc-algebra $(A,d)$ is called \textit{formal} if there exists a zig-zig of quasi-isomorphisms of dgc-algebras
	\[
	(H_*(A),0)\leftarrow \bigcdot \ ... \ \bigcdot \rightarrow (A,d).
	\]
\end{definition}

Let us prepair the $A_\infty$-algebra definition of formality. Let $(A,m)$ be an $A_\infty$-algebra, note that $m_2$ induces a map on homology
\[
m_2' : H_*(A)^{\otimes 2} \rightarrow H_*(A).
\]
We call this map the \textit{induced product}. It is easy to check that this is well-defined and moreover that $m_2'$ is associative. Hence $(H_*(A),0,m_2')$ is a dg-algebra.
\begin{definition}
	Let $(A,m)$ be an $A_\infty$-algebra. The following data is called a \textit{minimal model} for $A$: An $A_\infty$-structure on the homology $(H_*(A),m')$ with $m_1' = 0$ and $m_2'$ the induced product together with a quasi-isomorphism of $A_\infty$-algebras
	\[
	i: (H_*(A),m')\rightarrow (A,m),
	\]
	such that $i_1$ induces the identity map on homology.
\end{definition}
For the classical definition of minimal models of dgc-algebras we refer the reader to \cite{kaehler} where the theory of minimal models is presented based on Sullivan's paper \cite{sullivan_minimal_model} which was published two years later. The next theorem is proved in \cite{kadeishvili} for the case where $A$ is a dg-algebra. We give a proof in Section \ref{sec:pertub_lemma} (Theorem \ref{thm:existance_minimal_model}).

\begin{theorem} \label{thm:minimal_model}
	Let $(A,m)$ be an $A_\infty$-algebra then there exists a minimal model
	\[
	i: (H_*(A),m') \rightarrow (A,m).
	\]
	Moreover the $A_\infty$-structure is unique up to isomorphism.
\end{theorem}
Before we can proceed to define formality in the setting of $A_\infty$-algerbas we need to introduce some extra structure. Recall the definitions of a shuffle: Let $p,q$ be two natural numbers. An element of the permutation group $\sigma \in S_{p+q}$ is called a $(p,q)$\textit{-shuffle} if 
\[
\sigma(1)<...<\sigma(p),\quad \sigma(p+1)<...<\sigma(p+q).
\]
Denote the subgroup of $(p,q)$-shuffles by $Sh_{p,q}\subset S_{p+q}$. Let $A$ be some graded vector space. For $n=p+q$ define the shuffle product as follows
\begin{align*}
    \mu_{p,q}:A^{\otimes n} &\rightarrow A^{\otimes n}\\
    (x_1\otimes...\otimes x_n) &\mapsto \sum_{\sigma\in Sh_{p,q}} \mathrm{sgn}(\sigma)\epsilon(\sigma;x_1,...,x_n) x_{\sigma^{-1}(1)}\otimes... \otimes x_{\sigma^{-1}(n)}.
\end{align*}
Where $\epsilon(\sigma;x_1,...,x_n)$ is the Koszul sign (see \cite[Section 1.2]{markl_C_infty}).
\begin{definition}
    An $A_\infty$-algebra $(A,m)$ is called \textit{balanced}, if for all $n\geq 2$ $m_n$ vanishes on all $(p,q)$-shuffels for $p+q=n$, i.e.
    \[
    m_n \circ \mu_{p,q} = 0.
    \]
\end{definition}
Note that any dgc-algebra is a balanced $A_\infty$-algebra, i.e. balanced $A_\infty$-algebras can be seen as a generalization of dgc-algebras. Similarly we call an $A_\infty$-morphism $f$ between two balanced $A_\infty$-algebras \textit{balanced} if each $f_n$ disappears on shuffles. We call a minimal model \textit{balanced} if the structure on the homology and the inclusion morphism is balanced. With these definitions we can proceed to define formality:
\begin{definition}
	A balanced $A_\infty$-algebra $(A,m)$ is called \textit{formal} if there exits a balanced minimal model $i: (H_*(A),m') \rightarrow (A,m)$ with $m_n' = 0$ for $n\geq 3$.
\end{definition}
The following theorem tells us that this definition is a well-defined extension of the classical concept of formality.
\begin{theorem} \label{thm:formality_def_equivalenz}
	Let $(A,d)$ be a dgc-algebra, then the following are equivalent
	\begin{enumerate}
		\item The dgc-algebra $(A,d)$ is formal.
		\item There is a small zig-zag of quasi-isomorphisms of dgc-algebras
		\[
		((H_*(A),0)\leftarrow  \bigcdot \rightarrow (A,d).
		\]
		\item The balanced $A_\infty$-algebra $(A,d)$ is formal.
	\end{enumerate}
\end{theorem}
For a full proof see \cite[Theorem 11.4.9]{loday_vallette}. Later in Remark \ref{rmk:1->2} we give a partial proof of 1. implies 3..

\vspace{4mm}

In the remainder of this section we show formality of compact K\"ahler manifolds in the $A_\infty$-algebra setting. Let $M$ be a complex manifold. Denote by $\Omega^r(M)$ the space of complex valued differential $r$-forms on $M$, and by $\Omega(M) = \bigoplus_{r\leq 0} \Omega^{-r}(M)$ the graded vector space of complex valued differential forms. The differential $d$ and the wedge product $\wedge$ make this into a dgc-algebra. Denote by $d_c$ the complex conjugate of the differential $d$. Let $D\in \{d,d_c\}$ and define the subspace of \textit{harmonic forms} as the kernel of the corresponding Laplacian operator:
\[
\mathcal{H}_D(M):= \{\vartheta\in \Omega(M): \Delta_D(\vartheta) := (DD^* + D^*D)(\vartheta)=0\}.
\]
Recall that we can split the space of differential $r$-forms into the space of differential $(p,q)$ forms $\Omega^{p,q}(M)$, for $p+q=r$, i.e. $\Omega^r(M)= \bigoplus_{p+q=r}\Omega^{p,q}(M)$. Similarly, the subspaces of harmonic $r$-forms and harmonic $(p,q)$-forms as
\[
\mathcal{H}_D^r(M):= \mathcal{H}_D(M)\cap\Omega^r(M),\quad \mathcal{H}_D^{p,q}(M):= \mathcal{H}_D(M)\cap\Omega^{p,q}(M).
\]
The following theorem does not yet require the metric of $M$ to be K\"ahler. We refer the reader to \cite{griffiths} for a detailed proof.
\begin{theorem} \textnormal{\cite{griffiths}}
    Let $M$ be a compact complex manifold. The space of harmonic $(p,q)$-forms is finite-dimensional. Thus the projection onto the harmonic subspace
    \[
    \mathcal{H}_D^{p,q}:\Omega^{p,q}(M)\rightarrow \mathcal{H}_D^{p,q}(M)
    \]
    is well-defined, and there exists a unique operator, called the Green's operator,
    \[
    G_D:\Omega^{p,q}(M)\rightarrow \Omega^{p,q}(M)
    \]
    with $G_D(\mathcal{H}_D^{p,q}(M))=0$, $G_DD=DG_D$, and $G_DD^*=D^*G_D$ such that for any $\alpha\in \Omega^{p,q}(M)$
    \begin{equation} \label{eq:hodge_theorem}
        \alpha = \mathcal{H}_D^{p,q}(\alpha) + \Delta_DG_D(\alpha).
    \end{equation}
\end{theorem}
Equation \eqref{eq:hodge_theorem} is called the \textit{Hodge decomposition} of $\alpha$. The Hodge decomposition implies an isomorphism between the de Rahm cohomology and the harmonic subspace:
\begin{equation} \label{eq:harmonic=cohomology}
    \mathcal{H}_D(M)\cong H_D^*(M).
\end{equation}

Formality of a manifold refers to formality of it's de Rahm complex:
\begin{definition}
	Let $M$ be a complex manifold, we say $M$ is \textit{formal} if $(\Omega(M),d)$ is formal.
\end{definition}
Let $M$ be a compact K\"ahler manifold. From \cite{griffiths} we know that $\Delta_d=\Delta_{d_c}$ and hence the harmonic subspaces coincide $ \mathcal{H}_d(M)= \mathcal{H}_{d_c}(M) =: \mathcal{H}(M)$. Let $h:= d^*G_d$ and let $p:\Omega(M)\rightarrow H_{d_c}^*(M)$ and $i:H_{d_c}^*(M)\rightarrow\Omega(M)$ be the obvious projection and inclusion maps, then we conclude that
\[	
\begin{tikzcd}
		\arrow[loop left, "h"] (\Omega(M),d) \arrow[r, shift left, "p"] & \arrow[l, shift left, "i"](H_{d_c}^*(M),0)
\end{tikzcd}
\]
is a deformation retract\footnote{See Section \ref{sec:pertub_lemma} for the definition of a deformation retract.}. From Section \ref{sec:explicit_proofs} we get an $A_\infty$-structure on the cohomology: Formally set $\lambda_1:= -h^{-1}$, $\lambda_2 := \wedge$ is the wedge product, and inductively set
\begin{equation}
	\lambda_n := \sum_{\substack{k+l=n\\k,l\geq 1}} (-1)^{k(l+1)}\wedge(h\lambda_k\otimes h\lambda_l).
\end{equation}
Theorem \ref{thm:merkulov} tells us that $m_1=0$ and $m_n:= p\lambda_ni^{\otimes n}$ for $n\geq 2$ defines an $A_\infty$-structure and $i_1:=i$, $i_n:= \lambda_n i^{\otimes n}$ a minimal model\footnote{Because the space of differential forms is graded commutative the resulting minimal model is balanced for a full proof of this see \cite{markl_C_infty}.}:
\[
i:(H_{d_c}^*(M),m_n)\rightarrow (\Omega(M),d).
\]
\begin{theorem}
	Let $M$ be a compact K\"ahler manifold and let $i_n$ be defined as above then
    \[
    i:(H_{d_c}^*(M),0)\rightarrow (\Omega(M),d).
    \]
    is a minimal model and hence $M$ is formal as a balanced $A_\infty$-algebra.
\end{theorem}
\begin{proof}
By the discussion before the theorem we must only show that $m_n = 0$ for all $n\geq 3$ or equivalently we must show that there exist maps $f_n:H_{d_c}^*(M)^{\otimes n}\rightarrow \Omega(M)$ such that
\begin{equation} \label{eq:compact_kahler=formal}
    h\lambda_n i^{\otimes n} = d_c f_n ,\quad \forall n\geq 3.
\end{equation}
First observe that $h\lambda_2i^{\otimes 2}$ is $d_c$-exact:
\begin{align*}
    h\lambda_2i^{\otimes 2}&= d^*G_d \wedge i^{\otimes 2}\\
    &= d^*G_d (\mathcal{H} +d_cG_{d_c}d_c^* +d_c^*G_{d_c}d_c)\wedge i^{\otimes 2}\\
    &= G_d\underbrace{d^*\mathcal{H}}_{=0}\wedge i^{\otimes 2} + d^*G_d d_cG_{d_c}d_c^* \wedge i^{\otimes 2} + d^*G_dd_c^*G_{d_c}d_c \wedge i^{\otimes 2}\\
    &= - d_c d^*G_d G_{d_c}d_c^* \wedge i^{\otimes 2} + d^*G_dd_c^*G_{d_c}\wedge \underbrace{(d_c \otimes \mathrm{id}+ \mathrm{id}\otimes d_c)(i\otimes i)}_{= d_ci\otimes i + i\otimes d_ci=0} \\
    &= d_c f_2.
\end{align*}
Where $f_2 =-d^*G_d G_{d_c}d_c^* i^{\otimes 2}$. Formally set $f_1:= d_c^{-1}i$. Let us now prove equation \eqref{eq:compact_kahler=formal} by induction. For $n=3$ \begin{align*}
    h\lambda_3 i^{\otimes 3} &= d^*G_d(i\wedge h\lambda_2i^{\otimes 2} + h\lambda_2i^{\otimes 2}\wedge i)\\
    &= -d_c d^*G_d(i\wedge f_2 + f_2\wedge i).
\end{align*}
Fix $n>3$ and assume equation \eqref{eq:compact_kahler=formal} holds for all $m$ with $3\leq m<n$ then we compute the following
\begin{align*}
    \lambda_n i^{\otimes n} &= \sum_{\substack{k+l=n\\k,l\geq 1}} (-1)^{k(l+1)}\wedge(h\lambda_k\otimes h\lambda_l)i^{\otimes n}\\
    &\overset{\mathrm{i.a.}}{=} \sum_{\substack{k+l=n\\k,l\geq 1}} (-1)^{k(l+1)}\wedge (d_cf_k\otimes d_cf_l)\\
    &= (-1)^{n} d_c \wedge (i \otimes f_{n-1}) + \sum_{\substack{k+l=n\\k> 1, l\geq 1}} (-1)^{k(l+1)}d_c\wedge (f_k\otimes d_cf_l).
\end{align*}
\end{proof}

\section{\texorpdfstring{$A_\infty$}{TEXT}-Algebras and Coalgebras}
\label{sec:quasi-cofree}
Let $A$ be a graded vector space. In this section show how $A_\infty$-algebra structures on $A$ are in one-to-one correspondence to to differentials on the quasi-cofree coalgebra $T^c(A)$. The section is based on \cite{loday_vallette}.

\subsection{Quasi-Cofree Coalgebras}
We build a coalgebra out of a vector space: Let $V$ be a $k$-vector space. Denote by $T^c(V)$ the vector space
\[
T^c(V) := k\oplus V \oplus V^{\otimes 2}\oplus ...\ .
\]
Equip this space with the ``deconcatenation'' product $\Delta:T^c(V) \rightarrow T^c(V)\otimes T^c(V)$ which on $V^{\otimes n}$ for decomposable elements is given by
\[
\Delta(v_1,...,v_n) := \sum_{i=0}^n (v_1,...,v_i)\otimes (v_{i+1},...,v_n)
\]
and $\Delta(1) := 1\otimes 1$. We define a counit $\epsilon :T^c(V) \rightarrow k$, which is the identity on $k$ and zero otherwise. It is easy to check that $(T^c(V),\Delta, \epsilon)$ is a coalgebra, called the \textit{tensor coalgebra}. $T^c(V)$ is coaugmented by the inclusion map, $i:k \rightarrow T^c(V)$. Let $(\bar{T}^c(V), \bar{\Delta})$ denote the reduced coalgebra. Note that the reduced product is given by
\[
\bar{\Delta}(v_1,...,v_n) = \sum_{i=1}^{n-1}(v_1,...,v_i)\otimes (v_{i+1},...,v_n).
\]
\begin{definition}
	A coaugmented coalgebra $C$ is called \textit{conilpotent} if for all $x\in \bar{C}$ there exists an $n\in \mathbb{N}$ such that $\bar{\Delta}_n(x) =0$.
\end{definition}
Let $\pi: T^c(V) \rightarrow V$ denote the projection onto the first component. Note that for any $\omega \in \bar{T}^c(V)$ the $n$-th component (for $n\geq 2$) is given by
\[
\omega_n = (\pi\otimes ... \otimes \pi) \bar{\Delta}^{n-1}(\omega) = (\pi\otimes ... \otimes \pi) {\Delta}_{n-1}(\omega).
\]
We prove that the tensor coalgebra satisfies the following universal property.
\begin{prop} \label{prop:universal_property}
	\textnormal{\cite{loday_vallette}} Let $C$ be a conilpotent coalgebra and $\varphi:C\rightarrow V$ a linear map with $\varphi(1) = 0$, then $\varphi$ extends uniquely to a coaugmented coalgebra morphism $\tilde{\varphi}:C \rightarrow T^c(V)$ such that the following diagram commutes
\[
\begin{tikzcd}
	C \arrow[rd, "\varphi"] \arrow[d, "\tilde{\varphi}", dashed] \\
	T^c(V) \arrow[r, "\pi"] & V.
\end{tikzcd}
\]
\end{prop}
\begin{proof}
	We first show uniqueness: Let $C$ be a coaugmented coalgebra with coaugmentation map $u:k\rightarrow C$ such that $\tilde{\varphi}:C \rightarrow T^c(V)$ is an extension as in the proposition. Note that $C$ is equal to $k_{\bigcdot}1 \oplus \bar{C}$ as a vector space. Since morphisms of coaugmented coalgebras commute with the coaugmentations, we have
	\begin{equation}\label{eq:1}
		\tilde{\varphi}(1)=1.
	\end{equation}	
	Let $x\in \bar{C}$, then since $\tilde{\varphi}$ commutes with counits, we have
	\begin{equation} \label{eq:2}
		\tilde{\varphi}(x)_0=0.
	\end{equation}
	By compatibility with $\varphi$, we must have 
	\begin{equation} \label{eq:3}
		\tilde{\varphi}(x)_1=\varphi(x).
	\end{equation}
	Lastly, for $n\geq 2$ we compute
	\begin{align} \label{eq:4}
		\tilde{\varphi}(x)_n &=(\pi\otimes ... \otimes \pi) \bar{\Delta}^{n-1}(\tilde{\varphi}(x))\nonumber\\
		&=(\pi\otimes ... \otimes \pi) (\tilde{\varphi}\otimes ...\otimes \tilde{\varphi})\bar{\Delta}^{n-1}(x)\nonumber\\
		&=(\pi \tilde{\varphi}\otimes ... \otimes \pi \tilde{\varphi}) \bar{\Delta}^{n-1}(x)\nonumber\\
		&=(\varphi\otimes ... \otimes \varphi) \bar{\Delta}^{n-1}(x)\nonumber\\
		&=(\varphi\otimes ... \otimes \varphi) \sum_{(x)}x^{(1)}\otimes ... \otimes x^{(n)}\nonumber\\
		&=\sum_{(x)}\varphi(x^{(1)})\otimes ... \otimes \varphi(x^{(n)}).
	\end{align}
	From equations \eqref{eq:1}-\eqref{eq:4} we see that $\tilde{\varphi}$ is uniquely determined by $\varphi$. To show existence, we simply use equations \eqref{eq:1}-\eqref{eq:4} to define $\tilde{\varphi}$. It remains to check whether, for $x\in \bar{C}$ we have 
	\[
	\tilde{\varphi}(x) : = \sum_n (\varphi\otimes ... \otimes \varphi) \bar{\Delta}^{n-1}(x) \in T^c(V).
	\]
	This is ensured by the nilpotentcy of $C$, i.e.\ the sum is finite. 
\end{proof}
A coalgebra $C$ is called \textit{cofree}, if it is equal to $T^c(V)$ for some vector space $V$. Proposition \ref{prop:universal_property} tells us that a morphism $\psi$ from a conilpotent coalgebra $C$ into the cofree coalgebra $T^c(V)$ is uniquely determined by the first component $\pi\circ \psi$.

\vspace{4mm}

We have an analogous result in the graded framework. That is, if $(V,\{V_n\}_{n\in \mathbb{Z}})$ is a graded vector space. We equip the coalgebra $T^c(V)$ with a grading
\[
T^c(V)_n =\bigoplus_{i_1+...+i_r = n} V_{i_1}\otimes ... \otimes V_{i_r}.
\]
It is easy to check that the decomposition map is of degree $0$, thus $T^c(V)$ is a graded coalgebra. 
A graded coalgebra $C$ is called \textit{quasi-cofree} if it is equal to $T^c(V)$ for some graded vector space $V$. Before proving something about coderivations on $T^c(V)$ we need the following lemma.
\begin{lemma} \label{lem:d_respects_coaugmentation}
	Let $C$ be a coaugmented coalgebra and $d:C\rightarrow C$ a coderivation then the image of $d$ is contained in the kernel of the counit.
\end{lemma}
\begin{proof}
The coaumentation map induces a decomposition $C = k_{\bigcdot} 1\oplus \bar{C}$. This gives us a decomposition of the tensor product
\[
C\otimes C = \bar{C}\otimes \bar{C} \oplus \bar{C}\otimes k_{\bigcdot}1 \oplus k_{\bigcdot}1 \otimes \bar{C} \oplus k_{\bigcdot}1\otimes k_{\bigcdot}1.
\]
\textbf{Step 1}: We show that $d(1) \in \bar{C}$. Given the decomposition of $C$ we may write $d(1) = \alpha + \bar{c}$ for some $\alpha\in k$ and $\bar{c}\in \bar{C}$. Then on the one hand
\[
\Delta(d(1)) = \Delta(\alpha) + \Delta(\bar{c}) = \alpha 1\otimes 1+1\otimes \bar{c} + \bar{c}\otimes 1 +\bar{\Delta}(\bar{c}).
\]
On the other hand
\[
\Delta(d(1)) = (d\otimes \mathrm{id}+\mathrm{id}\otimes d)\underbrace{\Delta(1)}_{=1\otimes 1} = 2\alpha 1\otimes 1 + \bar{c}\otimes 1+1\otimes \bar{c} + \bar{c}\otimes \bar{c}.
\]
Combing the last two equations, we get
\[
\bar{\Delta}(\bar{c}) = \alpha 1\otimes 1 + \bar{c}\otimes \bar{c}.
\]
From the direct sum decomposition of the tensor product and the fact that $\bar{\Delta}(\bar{c})\in \bar{C}\otimes \bar{C}$ we get $\alpha=0$.

\vspace{4mm}

\textbf{Step 2}: We show that $d(\bar{c})\in \bar{C}$ for $\bar{c}\in \bar{C}$. Again write $d(\bar{c})=\alpha + \bar{x}$ for some $\alpha\in k$ and $\bar{x}\in \bar{C}$. We compute both sides of equation \eqref{eq:coderivation} as in Step 1:
\begin{align*}
		\alpha 1\otimes 1+1\otimes \bar{x} + \bar{x}\otimes 1 +\bar{\Delta}(\bar{x}) &= \Delta(d(\bar{x}))\\
		&= (d\otimes \mathrm{id} + \mathrm{id}\otimes d)\Delta(\bar{c})\\
		&= (d\otimes \mathrm{id} + \mathrm{id}\otimes d)(1\otimes \bar{c} + \bar{c}\otimes 1 - \bar{\Delta}(\bar{c}))\\
		&= 2\alpha 1\otimes 1 + \bar{c}\otimes d(1) + d(1)\otimes \bar{c} + \bar{x}\otimes 1 + 1 \otimes \bar{x}\\ &\quad +\sum_{(\bar{x})}d(\bar{x}^{(1)}) \otimes \bar{x}^{(2)} +. \bar{x}^{(1)}\otimes d(\bar{x}^{(2)}).
\end{align*}
Comparing terms on both sides, it follows that $\alpha = 0$.
\end{proof}
The next proposition is slightly more general than the version proved in \cite{loday_vallette} as our definition of a coderivation does not demand $d(1) = 0$.
\begin{prop} \label{prop:coderivations}
	Let $V$ be a vector space. Composition with the projection map induces an isomorphism
	\[
	\mathrm{Coder}(T^c(V)) \xrightarrow{\cong} \mathrm{Hom}(T^c(V), V').
	\]
\end{prop}
\begin{proof}
Notice that both sides are vector spaces, and the map is linear. Let $d$ be a coderivation on $T^c(V)$ such that $\pi\circ d = 0$. We show that $d= 0$: Given the decomposition $T^c(V)= k_{\bigcdot}1\oplus \bar{T}^c(V)$, we first show that $d$ is zero on $ \bar{T}^c(V)$. Let $x\in T^c(V)$, then from Lemma \ref{lem:d_respects_coaugmentation} we know that $d(x)\in \bar{T}^c(V)$. The first component of $d(x)$ is zero by assumption and for $n\geq 2$ we have
\begin{align*}
	d(x)_n &= (\pi\otimes ...\otimes \pi)\Delta_{n-1}(d(x))\\
	&= (\pi\otimes ...\otimes \pi) \sum_i (\mathrm{id}\otimes ... \otimes \mathrm{id}\otimes d\otimes \mathrm{id}\otimes ...\otimes \mathrm{id}) \Delta_{n-1} (x)\\
	&= \sum_i (\pi\otimes ... \otimes \pi\otimes \pi d\otimes \pi\otimes ...\otimes \pi) \Delta_{n-1} (x)\\
	&= 0.
\end{align*}
It follows that $d(x)=0$. Now let $b:T^c(V)\rightarrow V$ be an arbitrary linear map. Our goal is to define a coderivation with first component equal to $b$. From Lemma \ref{lem:d_respects_coaugmentation} we know that for any $y\in T^c(V)$ we must have $d(y)_0=0$. Set $d(1):= b(1)$, for $x\in \bar{T}^c(V)$ set $d(x)_1 := b(x)$,t and for $n\geq 2$ set 
\[
d(x)_n := \sum_i \sum_{(x)} \pi(x_{(1)})\otimes ... \otimes \pi(x_{i-1}) \otimes b(x_{(i)}) \otimes \pi(x_{(i+1)}) \otimes...\otimes \pi(x_{(n)}).
\]
We claim that $d(x):= \sum_{n\geq 1} d(x)_n$ is well-defined, i.e.\ the sum is finite. Notice that $b(1)$ might be non zero, so we cannot use the reduced coproduct to define $d$ and argue that the sum is finite by conilpotency. However, a similar argument works: Without loss of generality assume that $x$ is homogeneous $x = v_1\otimes ... \otimes v_m$, then we claim that $d(x)_n=0$ for all $n\geq m+2$. Indeed, for $n\geq m+2$ every summand of $\Delta_{n-1}(x)$ has at least two $1$s and $b(1)$ might be non-zero but $\pi(1) = 0$ so $d(x)_n=0$. It remains to show that $d$ satisfies the Leibniz rule, this can be verified by a direct calculation.
\end{proof}

Let us extend the definition of coderivations:
\begin{definition}
	Let $f:C\rightarrow C'$ a coalgebra morphism. A linear map $d:C\rightarrow C'$ is called a \textit{coderivation along $f$} if
	\begin{equation} \label{eq:leibniz}
		\Delta' \circ d = (d\otimes f)\circ \Delta + (f\otimes d) \circ \Delta.
	\end{equation}
\end{definition}
Denote by $\mathrm{Coder}_f(C,C')$ the $k$-vector space of coderivations along $f$. Note that in the case $C=C'$ and $f= \mathrm{id}_C$ a coderivation along the identity is just a coderivation in the usual sense. The following proposition is proved similarly as Proposition \ref{prop:coderivations}.

\begin{prop} \label{prop:coderivations_along_f}
Let $f:T^c(V)\rightarrow T^c(V')$ be a coaugmented coalgebra morphism between two cofree coalgebras. Composition with the projection map induces a bijection
\[
 \mathrm{Coder}_f(T^c(V),T^c(V')) \xrightarrow{\cong} \mathrm{Hom}(T^c(V), V').
\]
\end{prop}

\subsection{\texorpdfstring{$A_\infty$}{TEXT}-algebras and Coalgebras}
Let $(A,m)$ be an $A_\infty$-algebra. Recall that the maps $m_n:A^{\otimes n} \rightarrow A$ are of degree $n-2$. Recall that the suspension map $s:A\rightarrow sA$ is of degree $+1$. We shift the degree of the $m_n$ maps such that we get maps of degree $-1$: Set $b_n := s\circ m_n \circ (s^{-1})^{\otimes n}$ such that the following diagram commutes
\[
\begin{tikzcd}
	(sA)^{\otimes n} \arrow[r, "b_n"] &sA\\
	A^{\otimes n} \arrow[u, "s^{\otimes n}"] \arrow[r, "m_n"'] &A. \arrow[u, "s"']
\end{tikzcd}
\]
By Proposition \ref{prop:coderivations} the map $\sum_{n\geq 1} b_n : T^c(sA) \rightarrow sA$ uniquely lifts to a coderivation $d:T^c(sA)\rightarrow T^c(sA)$. The degree of this coderivation is $-1$. The $A_\infty$-structure ensures that this coderivation squares to zero:
\begin{prop} \label{prop:(A,m)==d^2} Let $A$ be a graded vector space. There is a one-to-one correspondence between $A_\infty$-structures $A$ and coderivations $d$ on $T^c(sA)$ of degree $-1$ such that $d(1) = 0$ and $d^2 = 0$.
\end{prop} 
\begin{proof}
Let $(A,m)$ be an $A_\infty$-algebra and $d$ be the 
coderivation on $T^c(sA)$ defined above. Then $d^2$ is again a coderivation as
\begin{align*}
	\Delta d^2 &= (d\otimes \mathrm{id} +\mathrm{id}\otimes d) (d\otimes \mathrm{id} +\mathrm{id}\otimes d) \Delta \\
	&= (d^2\otimes \mathrm{id} +d\otimes d - d\otimes d +\mathrm{id}\otimes d^2) \Delta = (d^2\otimes \mathrm{id} +\mathrm{id}\otimes d^2) \Delta.
\end{align*}
The minus comes from the Koszul sign convention. Thus, by Proposition \ref{prop:coderivations}, it suffices to show that $\pi \circ d^2 = 0$. From the proof of Proposition \ref{prop:coderivations} we get on $(sA)^{\otimes n}$
\begin{align*}
	\pi \circ d^2 &= \sum_{n=r+j+t} b_{r+t+1} (\mathrm{id}^r\otimes b_j\otimes \mathrm{id}^t)\\
	&= \sum_{n=r+j+t}(-1)^t s m_{r+t+1} ((s^{\otimes r})^{-1}\otimes \underbrace{s^{-1}b_j}_{= m_j (s^{-1})^{\otimes j}} \otimes (s^{\otimes t})^{-1})\\
	&= s \Big[\sum_{n=r+j+t} (-1)^{rj +t} m_{r+t+1} (\mathrm{id}^r \otimes m_{j} \otimes \mathrm{id}^t)\Big] (s^{\otimes n})^{-1} = 0.
\end{align*}
Lastly, $d(1)$ is zero by definition. Conversely, assume that $d$ is a coderivation which squares to zero and $d(1)=0$. Denote by $b_n$ the restriction of $\pi\circ d$ onto $(sA)^{\otimes n}$. Then $m_n := (-1)^{\frac{n(n-1)}{2}}(s)^{-1} b_n s^{\otimes n}$ defines an $A_\infty$-structure since $(s)^{-1}( \pi \circ d^2) s^{\otimes n} = 0$ are exactly the $A_\infty$-relations.
\end{proof}

Let $(A,m)$ be an $A_\infty$-algebra and $d$ the corresponding coderivation on $T^c(sA)$. We call $(T^c(sA), d)$ the \textit{differential graded coalgebra associated to} $(A,m)$.
\begin{prop} \label{prop:temp_coder}
	Let $V,V'$ be two graded vector spaces. Let $F:T^c(V)\rightarrow T^c(V')$ be a morphism of conilpotent graded coalgebras of degree $0$. Let $d$ be a coderivation on $T^c(V)$ and $d'$ a coderivation on $T^c(V')$. Let $F_n$, $b_n$ denote\footnote{Here $n\geq 0$, so technically $F_0$ is also to be considered in the sum \eqref{eq:something_something_dark_side}, but as $\pi(F(1)) = \pi(1)=0$ these terms are zero regardless, so there is no ambiguity.} to restrictions of $\pi\circ F$, $\pi\circ d$, respectively, to $V^{\otimes n}$, and $b_n'$ the restriction of $\pi' \circ d'$ to $V'^{\otimes n}$. Then the following are equivalent
	\begin{enumerate}
		\item $d'F = Fd$.
		\item $d'(1) = F(d(1))$ and for all $n\geq 1 $ 
		\begin{equation} \label{eq:something_something_dark_side} \sum_{i_1+...+i_r = n} b_r' (F_{i_1}\otimes ...\otimes F_{i_r}) = \sum_{r+s+t = n} F_{r+t+1} (\mathrm{id}^{\otimes r}\otimes b_s \otimes \mathrm{id}^{\otimes t}).\end{equation}
	\end{enumerate}
\end{prop}

\begin{proof}
	Assume that $d'F = F d$. From the proof of Proposition \ref{prop:universal_property} we see that 
	\[
	F = \sum_n \sum_{i_1+...+i_r=n} F_{i_1}\otimes ...\otimes F_{i_r}.
	\]
	From the proof of Proposition \ref{prop:coderivations} we see that
	\[
	d = \sum_{n\geq 0} \sum_{r+s+t=n} \mathrm{id}^{\otimes r}\otimes b_s\otimes \mathrm{id}^{\otimes t}.
	\]
	For $n\geq 1$ the equation $\pi d'F = \pi F d$ restricted to $V^{\otimes n}$ is precisely equation \eqref{eq:something_something_dark_side}. Conversely, assume $F$ satisfies equation \eqref{eq:something_something_dark_side}. First note that $d'F -Fd$ is a coderivation along $F$. Indeed,
	\begin{align*}
		 \Delta'(d'F -Fd) &= (\mathrm{id}\otimes d' + d'\otimes \mathrm{id})\Delta' F - (F\otimes F )\Delta d\\
		 &= (\mathrm{id}\otimes d' + d'\otimes \mathrm{id})(F\otimes F )\Delta - (F\otimes F )(\mathrm{id}\otimes d + d\otimes \mathrm{id})\Delta \\
		 &= (F\otimes (d'F-Fd) + (d'F-Fd)\otimes F)\Delta.
	\end{align*}
	Hence by Proposition \ref{prop:coderivations_along_f} the map $d'F -Fd$ is uniquely determined by its first component $\pi\circ (d'F -Fd)$ which is zero by assumption.
\end{proof}

\begin{corollary} \label{cor:A_infty_morphisms=F}
Let $(A,m)$ and $(A',m')$ be two $A_\infty$-algebras with corresponding dg-coalgebras $(T^c(sA),d)$, $(T^c(sA'),d')$, respectively. There is a one-to-one correspondence between morphisms of $A_\infty$-algebras $f:(A,m)\rightarrow (A',m')$ and morphisms of dg-coalgebras $F:(T^c(sA),d)\rightarrow (T^c(sA'),d')$ of degree $0$.
\end{corollary}
\begin{proof}
Assume $f:(A,m)\rightarrow (A',m')$ is a morphism of $A_\infty$-algebras. As before set $b_n:= sm_n(s^{-1})^{\otimes n}$ and $b_n':= sm_n'(s^{-1})^{\otimes n}$. From Proposition \ref{prop:(A,m)==d^2} we know that $\sum_n b_n$, $\sum_nb_n'$ extend to coderivations $d, d'$ on $T^c(sA)$, $T^c(sA')$, respectively. Set $F_n := sf_n(s^{-1})^{\otimes n}$. From Proposition \ref{prop:universal_property} we know that $\sum_n F_n$ extends to a unique coalgebra morphism $F:T^c(sA)\rightarrow T^c(sA')$. It remains to show that $F$ commutes with the differentials. We use Proposition \ref{prop:temp_coder}: Since $F(1)=1$ and $d(1)=d'(1)=0$ it suffices to prove equation \eqref{eq:something_something_dark_side}:
\begin{align*}
	&\sum_{r+j+t=n} F_{r+t+1}(\mathrm{id}^{\otimes r}\otimes b_j\otimes \mathrm{id}^{\otimes t}) \\
	&= s\sum_{r+j+t=n} f_{r+t+1}(s^{-1})^{\otimes r+t+1}(\mathrm{id}^{\otimes r}\otimes b_j\otimes \mathrm{id}^{\otimes t})\\
	&=s\sum_{r+j+t=n} (-1)^{t} f_{r+t+1}((s^{-1})^{\otimes r}\otimes \underbrace{s^{-1} b_j}_{= m_j(s^{-1})^{\otimes j}} \otimes (s^{-1})^{\otimes t})\\
	&=s\sum_{r+j+t=n} (-1)^{rj+t} f_{r+t+1}(\mathrm{id}^{\otimes r}\otimes m_j \otimes \mathrm{id}^{\otimes t}) (s^{-1})^{\otimes n}\\
	&=s\sum_{i_1+...+i_r= n} (-1)^l m_r(f_{i_1}\otimes ... \otimes f_{i_r})(s^{-1})^{\otimes n}\\
	&= ...= \sum_{i_1+...+i_r= n} b_r'(F_{i_1}\otimes ... \otimes F_{i_r}).
\end{align*}
The converse is proved by reversing these steps.
\end{proof}

Now we are finally able to prove that equation \eqref{eq:a_infty_composition} is well-defined:
\begin{prop} \label{prop:composition}
Composition of $A_\infty$-morphims is well-defined.
\end{prop}

\begin{proof}
Let $f:(A,m) \rightarrow (A',m')$, $f':(A',m')\rightarrow (A'',m'')$ be two $A_\infty$-morphisms. Denote by $F, F'$ the corresponding morphisms of differential graded coalgebras, then $F'\circ F$ is again a morphism of differential graded coalgebras, and hence defines a unique $A_\infty$-morphism $g:(A,m)\rightarrow (A'',m'')$ given by
\begin{align*}
	g_n &:= (-1)^{\frac{n(n-1)}{2}} s ^{-1} (F'\circ F)_n s^{\otimes n} \\
	&= s ^{-1}\big[\sum_{i_1+...+i_r = n} F'_r(F_{i_1}\otimes...\otimes F_{i_r} ) \big]s^{\otimes n}\\
	&= (-1)^{\frac{n(n-1)}{2}} \big[\sum_{i_1+...+i_r = n} f'_r(s^{-1})^{\otimes r} (F_{i_1}\otimes...\otimes F_{i_r} ) \big]s^{\otimes n}\\
	&= (-1)^{\frac{n(n-1)}{2}} \big[\sum_{i_1+...+i_r = n} f'_r(s^{-1}F_{i_1}\otimes...\otimes s^{-1}F_{i_r} ) \big]s^{\otimes n}\\
	&= (-1)^{\frac{n(n-1)}{2}} \big[\sum_{i_1+...+i_r = n} f'_r(f_{i_1}(s^{-1})^{\otimes i_1}\otimes...\otimes f_{i_r}(s^{-1})^{\otimes i_r} ) \big]s^{\otimes n}\\
	&= (-1)^{\frac{n(n-1)}{2}} \big[\sum_{i_1+...+i_r = n} (-1)^l f'_r(f_{i_1}\otimes...\otimes f_{i_r} ) \big](s^{-1})^{\otimes n}s^{\otimes n}\\
	&=\sum_{i_1+...+i_r = n} (-1)^l f'_r(f_{i_1}\otimes...\otimes f_{i_r} )\\
	&= (f'\circ f)_n.
\end{align*}
We conclude that $f'\circ f	$ is a morphism of $A_\infty$-algebras.
\end{proof}

Also we are able to prove Proposition \ref{prop:a_infty_isomorphism}.
\begin{prop} \label{prop:inverse}
    A morphism of $A_\infty$-algebras $f:(A,m)\rightarrow (A',m')$ is an isomorphism if and only if $f_1$ is an isomorphism.
\end{prop}

\begin{proof}
	Let $f:(A,m)\rightarrow (A',m')$ be an $A_\infty$-algebra morphism such that $f_1$ is invertible. Let $F:(T^c(sA),d)\rightarrow (T^c(sA'),d')$ be the corresponding morphism of dg-coalgebras. Let $g_n:A'^{\otimes n} \rightarrow A$ be the maps defined as in \eqref{eq:inverse}
	\[
	g_n := f_1^{-1}\Big(\sum_{\substack{n=i_1+...+i_r\\r>1}} (-1)^{l} f_r(g_{i_1}\otimes ... \otimes g_{i_r})\Big).
	\]
	We show that this is in fact an $A_\infty$-algebra morphism: The map
	\[
	\sum _n s g_n (s^{-1})^{\otimes n}: T^c(sA)\rightarrow sA'
	\]
	lifts to an coalgebra morphism $G: T^c(sA)\rightarrow T^c(sA')$. We know that the coalgebra morphism $F \circ G$ corresponds to the linear map $\sum_n s^{-1} (f\circ g)_n s^{\otimes n}$, and hence we know that $G$ is a right-inverse of $F$. Conversely, define $g_1' := f_1^{-1}$, and then inductively for all $n$
	\[
	g_n' = \sum_{\substack{i_1+...+i_r = n\\r<n}} (-1)^l g_r' (f_{i_1}\otimes ... \otimes f_{i_r})((f_1)^{-1}\otimes ...\otimes (f_1)^{-1}).
	\]
	Then the map $\sum_n s^{-1} g_n' s^{\otimes n}: T^c(sA') \rightarrow sA$ defines a coalgebra morphism $G' :T^c(sA') \rightarrow T^c(sA)$ which is by construction a left-inverse of $F$. And hence $G$ and $G'$ are equal, that is 
	\[
	G = \mathrm{id}_{T^c(sA)} \circ G = (G' \circ F) \circ G = G' \circ (F \circ G) = G' \circ \mathrm{id}_{T^c(sA')} = G'.
	\]
	It remains to check that $G$ is a morphism of differential coalgebras, and indeed we compute
	\[
	G d' = Gd'FG = GFdG = dG.
	\]
	We conclude that $g$ is a morphism of $A_\infty$-algebras. 
	\end{proof}
	
\section{Perturbation Lemma} \label{sec:pertub_lemma}

One goal in this section is to prove the hard direction (1.\ implies 3.) of Theorem \ref{thm:formality_def_equivalenz} via the perturbation lemma.
\subsection{Deformation retracts}
Recall the following basic definitions:
\begin{definition}
	A \textit{chain complex (of vector spaces)} is a graded vector space $A$ together with a linear map $d:A\rightarrow A$ of degree $-1$ such that $d^2=0$. $d$ is called the \textit{differential}.
\end{definition}
\begin{definition}
	 A \textit{chain map} $f:A\rightarrow B$ is a linear map of degree $0$ which commutes with the differentials.
\end{definition}
\begin{definition}
	Two chain maps $f,g :A\rightarrow B$ are called \textit{chain homotopic} if there exists a linear map $h:A\rightarrow B$ of degree $+1$ such that
	\[
	 f - g = d_B \circ h + h \circ d_A.
	\]
	$h$ is called a \textit{chain homotopy from $f$ to $g$} and we write $f\simeq g$.
\end{definition}
 Consider the following special case of a homotopy equivalences.
\begin{definition}
    Let $A,B$ be two chain complexes. $B$ is called a \textit{deformation retract} of $A$ if there exist chain maps $i:B\rightarrow A$ and $p:A \rightarrow B$ such that $\mathrm{id}_B= p\circ i$ and $\mathrm{id}_A \simeq i\circ p$.
\end{definition}
Let $B$ be a deformation retract of $A$. Let $h:A\rightarrow A$ be a chain homotopy from $\mathrm{id}_A$ to $i\circ p$, i.e.\
\[
\mathrm{id}_A - i\circ p = d_B \circ h + h \circ d_A.
\]
By slight abuse of notation we sometimes call the collection of chain complexes and maps
\[
\begin{tikzcd}
		\arrow[loop left, "h"] (A,d_A) \arrow[r, shift left, "p"] & \arrow[l, shift left, "i"](B,d_B)
\end{tikzcd}
\]
a deformation retract. We are in particular interested in the case where $h$ satisfies the additional conditions  
\begin{equation} \label{eq:additiona_conditions}
    h^2 = 0, \ p\circ h = 0\ \mathrm{and}\ h\circ i = 0.
\end{equation}
It is easy to prove that any homotopy can be deformed such that these additional conditions are satisfied.
\begin{prop} 	
    Let $A$ be a chain complex. The homology $H_*(A)$, equipped with the zero differential, is a deformation retract of $A$.
\end{prop}
\begin{proof}
Let $A = \bigoplus_{n\in\mathbb{Z}} A_n$ and recall that a map $d$ of degree $-1$ is equivalent to a family of linear maps 
\[
d_n : A_n \rightarrow A_{n-1}.
\]
Recall the following definitions 
\[
H_n(A) := \ker{d_n}/\mathrm{im}\ {d_{n+1}},\quad H_*(A) := \bigoplus_{n\in \mathbb{Z}} H_n(A).
\]
Let $A_n^{E}\subset A$ be the subspace of exact forms. By Zorn's Lemma there exists a subspace $A_n^C$ such that
\[
\ker{d_n} = A_n^{E}\oplus A_n^C.
\]
Again by Zorn's Lemma there exists a subspace $A_n^\perp$ such that 
\[
A_n= A_n^{E}\oplus A_n^C\oplus A_n^\perp.
\]
Note that $A_n^\perp$ can be realised as the image of $d_n$, and $H_n(A)$ can be identified with $A_n^C$. Thus $A_n$ can be decomposed as 
\[
A_n= A_n^{E}\oplus H_n(A) \oplus A_{n-1}^{E}.
\]
Let 
\[
p_n :A_n \rightarrow H_n(A), \quad i_n : H_n(A) \rightarrow A_n
\]
be the projection and inclusion maps. Set $h_n : A_n \rightarrow A_{n+1}$ to be the map that identifies $A_n^{E}\subset A_n$ with its copy in $A_{n+1}$. By construction these maps satisfy the following equations:
\begin{align*}
    p_n i_n &= \mathrm{id},\\
    \mathrm{id}-i_n p_n &= h_{n-1}d_n + d_{n+1}h_n,\\
    h_n i_n = 0,\ p_{n+1} h_n&= 0,\ \mathrm{and}\ h_{n+1}h_n = 0.
\end{align*}
Thus, adding up for all $n$ gives a deformation retract:
\[
p:= \bigoplus_{n\in \mathbb{Z}} p_n :A\rightarrow H_*(A), \quad i:= \bigoplus_{n\in \mathbb{Z}} i_n : H_*(A)\rightarrow A
\]
\[
h:= \bigoplus_{n\in \mathbb{Z}} h_n : A \rightarrow A.
\]
Moreover, the additional conditions \eqref{eq:additiona_conditions} are satisfied. 
\end{proof}

Let $C, D$ be two dg-coalgebras. Note that dg-coalgebras are in particular chain complexes. Assume that 
\[
\begin{tikzcd}
    \arrow[loop left, "h"] (C,d_C) \arrow[r, shift left,"p"] & \arrow[l, shift left, "i"] (D, d_D)
\end{tikzcd}
\]
is a deformation retract. If $i$ and $p$ are morphisms of dg-coalgebras we say that the deformation retract is a \textit{deformation retract of dg-coalgebras}.

\subsection{Perturbation Lemma for Coalgebras}
Fix a dg-algebra $(A,d,\nu)$. Since dg-algebras are chain complexes, there exists a deformation retract
\[
\begin{tikzcd}
		\arrow[loop left, "h"] (A,d) \arrow[r, shift left, "p"] & \arrow[l, shift left, "i"](H_*(A),0)
\end{tikzcd}
\]
such that the additional conditions \eqref{eq:additiona_conditions} are satisfied. Note that $i$ and $p$ are only chain maps -- not morphisms of dg-algebras. Let $T^c(sA)$, $T^c(sH_*(A))$ denote the quasi-cofree coalgebras over $sA$, $sH_*(A)$, respectively. $i$ and $p$ induce coalgebra maps as follows: Define a map $T^c(sA) \rightarrow sH_*(A)$ equal to $s\circ p \circ s^{-1}$ on $sA$ and zero else. This map lifts to a coalgebra map of degree $0$
\[
P:T^c(sA) \rightarrow T^c(sH_*(A)).
\]
Similarly, $s\circ i\circ s^{-1}$ lifts to a coalgebra map of degree $0$
\[
I: T^c(sH_*(A))\rightarrow T^c(sA).
\]
Let us forget the product on $A$ for a moment and think of $(A,d)$ simply as a chain complex. In particular, the chain complex $(A,d)$ is an $A_\infty$-algebra, and thus by Proposition \ref{prop:(A,m)==d^2} there exists a square zero coderivation $\tilde{d}$ on $T^c(sA)$. $(T^c(sA), \tilde{d})$ is again a chain complex. We give $T^c(H_*(A))$ the zero differential and claim that 
\begin{equation}	\label{eq:tikz_1}
\begin{tikzcd} 
		(T^c(sA),\tilde{d}) \arrow[r, shift left, "P"] & \arrow[l, shift left, "I"](T^c(H_*(sA)),0).
\end{tikzcd}
\end{equation}
are cochain maps. Note that since $P$ is a coalgebra morphism of degree $0$ and $\tilde{d}$ is a coderivation $P \tilde{d}$ is a coderivation along $P$. So $P\tilde{d}$ is uniquely determined by it's first component $\pi P \tilde{d}$, which is zero because $p d = 0$. Similarly $\tilde{d}I$ is a coderivation along $I$ and the first component vanishes because $d i = 0$. Our goal is to make \eqref{eq:tikz_1} into a deformation retract. So we need a homotopy on $T^c(sA)$: We define a family of maps\footnote{These are not called $h_n$ because the name is reserved for later.} $\tilde{h}_n : sA^{\otimes n} \rightarrow sA^{\otimes n}$ by $\tilde{h}_1 := s h s^{-1}$ and for $n\geq 2$
\[
\tilde{h}_n := shs^{-1}\otimes sips^{-1}\otimes ...\otimes sips^{-1}+ ... + \mathrm{id}\otimes ...\otimes \mathrm{id} \otimes shs^{-1}.
\]
Set $H := \sum \tilde{h}_n : T^c(sA)\rightarrow T^c(sA)$, where $\tilde{h}_0 := 0: k \rightarrow k$. $H$ is well-defined by conilpotency. Moreover, $H$ is a morphism of graded vector spaces of degree $+1$. Let us prove that
\[
\begin{tikzcd}
		\arrow[loop left, "H"] (T^c(sA),\tilde{d}) \arrow[r, shift left, "P"] & (T^c(sH_*(A)),0) \arrow[l, shift left, "I"]
\end{tikzcd}
\]
is a deformation retract of dg-coalgebras. We already know that $P$ and $I$ are coalgebra maps that commute with the differentials. It remains to prove the following:
\begin{enumerate}
	\item $P I = \mathrm{id}$.
	\item $ \mathrm{id}-IP = \tilde{d}H + H\tilde{d}$.
	\item $P H =0$, $H I = 0$, $H^2 = 0$.
\end{enumerate}
1.\ follows from $pi = \mathrm{id}$. Before getting started on the proof of 2.\ and 3.\ we need to extend our definition of coderivations along maps.
\begin{definition}
	Let $f,f':C\rightarrow C'$ be two morphisms of coalgebras. A linear map $d:C\rightarrow C'$ is called a \textit{coderivation along} $(f,f')$ if 
	\[
	\Delta'd = (f\otimes d + d\otimes f') \Delta.
	\]
\end{definition}
Denote by $\mathrm{Coder}_{(f,f')}$ the space of coderivations along $(f,f')$. As expected, Proposition \ref{prop:coderivations_along_f} generalizes to the new definition of coderivations along two maps:
\begin{prop} \label{prop:coderivations_along_(f,f')}
	Let $f,f':T^c(V)\rightarrow T^c(V')$ be two coaugmented coalgebra morphism between cofree coalgebras. Composition with the projection map induces a bijection
	\[
	\mathrm{Coder}_{(f,f')}(T^c(V),T^c(V')) \rightarrow \mathrm{Hom}(T^c(V), V').
	\]
\end{prop}
The proof is again similar to the proof of Proposition \ref{prop:coderivations}. With this we can proceed to prove 2.\ and 3..\ We start of by showing that $H$ is a coderivation along $(\mathrm{id}, IP)$. That is we must prove 
\begin{equation} \label{eq:H_a_coderivation}
	\Delta H= (\mathrm{id} \otimes H + H\otimes IP) \Delta.
\end{equation}
On $k$ this is clear. On each $(sA)^{\otimes n}$ the proof is by induction on $n$: For $n=1$, let $x\in sA$ and compute the following
\begin{align*}
	(\mathrm{id} \otimes H + H\otimes IP)\Delta (x)&= (\mathrm{id} \otimes H + H\otimes IP)(1\otimes x +x\otimes 1)\\
	&= 1\otimes h(x)+ h(x)\otimes 1 \\
	&= \Delta H(x).
\end{align*}
Assume equation \eqref{eq:H_a_coderivation} is true on $(sA)^{\otimes n}$ for a fixed $n\geq 1$. To make things more clear denote the tensor product $T^c(sA) \boxtimes T^c(sA)$ with a different symbol. We compute
\begin{align*}
	\bar\Delta \tilde{h}_{n+1} &= \bar\Delta (\tilde{h}_n\otimes sips^{-1} + \mathrm{id}^{\otimes n} \otimes shs^{-1})\\
	&= \sum_{j=1}^{n-1} \mathrm{id}^{\otimes j} \boxtimes \tilde{h}_{n-j}\otimes sips^{-1} + \tilde{h}_j \boxtimes (sips^{-1})^{\otimes n-j} \otimes sips^{-1} \\
	&\qquad +\mathrm{id}^{\otimes j} \boxtimes \mathrm{id}^{\otimes n -j} \otimes shs^{-1} + h_n \boxtimes sips^{-1} + \mathrm{id}^{\otimes n} \boxtimes h \\
	& = \sum_{j=1}^{n-1} \mathrm{id}^{\otimes j} \boxtimes (\tilde{h}_{n-j}\otimes sips^{-1}+\mathrm{id}^{\otimes n -j} \otimes shs^{-1}) + \mathrm{id}^{\otimes n} \boxtimes shs^{-1} \\
	&\qquad + \sum_{j=1}^{n} \tilde{h}_j \boxtimes (sips^{-1})^{\otimes n+1-j}\\
	& = \sum_{j=1}^{n} \mathrm{id}^{\otimes j} \boxtimes (\tilde{h}_{n-j+1}) +\sum_{j=1}^{n} \tilde{h}_j \boxtimes (sips^{-1})^{\otimes n+1-j} \\
	&= (\mathrm{id}\boxtimes H + H\boxtimes IP)\bar{\Delta}.
\end{align*}
Thus $H$ is a coderivation along $(\mathrm{id},IP)$ of degree $+1$. Since $\tilde{d}$ is a coderivation along the identity of degree $-1$ it follows that $\tilde{d}H +H\tilde{d}$ is a coderivation along $(\mathrm{id},IP)$ of degree $0$. Thus $\tilde{d}H +H\tilde{d}$ is uniquely determined by it's first component which is equal to $shds^{-1}+sdhs^{-1}$ on $sA$ and zero else. However, $\mathrm{id}-IP$ is also a coderivation along $(\mathrm{id},IP)$ of degree $0$. Indeed, we compute
\begin{align*}
	\Delta (\mathrm{id}-IP) &= (\mathrm{id}\otimes \mathrm{id}- IP\otimes IP) \Delta\\
	&= (\mathrm{id}\otimes \mathrm{id}-IP\otimes IP +\mathrm{id} \otimes IP - \mathrm{id} \otimes IP+ ) \Delta\\
	&= (\mathrm{id}\otimes (\mathrm{id}-IP) + (\mathrm{id}-IP) \otimes IP) \Delta.
\end{align*}
The first component of $\mathrm{id}-IP$ is equal to the first component of $\tilde{d}H +H\tilde{d}$. Thus by uniqueness they are equal
\begin{equation}\label{eq:deformation_retract_dgca}
    \mathrm{id}-IP = \tilde{d}H +H\tilde{d}.
\end{equation}
Lastly, let us prove 3.: $PH, HI$ are a coderivation along $P, I$, respectively. They both vanish because their first component due to $ph = 0$ and $hi = 0$. $H^2$ is a coderivation along $(\mathrm{id}, IP)$. Indeed, we compute
\begin{align*}
	\Delta H^2 &= (\mathrm{id}\otimes H + H\otimes IP)(\mathrm{id}\otimes H + H\otimes IP)\Delta\\
	&= (\mathrm{id}\otimes H^2 - H\otimes HIP +H\otimes IPH + H^2\otimes IPIP)\Delta\\
	&= (\mathrm{id}\otimes H^2+ H^2\otimes IP)\Delta
\end{align*}
where the last line uses $HI = 0$, $PH = 0$ and $PI = \mathrm{id}$. The first component of $H^2$ vanishes again because of $h^2 = 0.$ This concludes the proof that 
\[
\begin{tikzcd}
	\arrow[loop left, "H"] (T^c(sA),\tilde{d}) \arrow[r, shift left, "P"] & (T^c(sH_*(A)),0) \arrow[l, shift left, "I"]
\end{tikzcd}
\]
is a deformation retract of dg-coalgebras. We now introduce a perturbation to the differential $\tilde{d}$: The dg-algebra $(A,d,\nu)$ is of course also an $A_\infty$-algebra and thus by Proposition \ref{prop:(A,m)==d^2} there exists a differential $D$ on $T^c(sA)$ of degree $-1$. Define
\[
t:= D-\tilde{d}, \mathrm{such}\ \mathrm{that}\ D = \tilde{d}+t.
\]
Since $D,\tilde{d}$ are both differentials so is $t$. $t$ is called the perturbation of the differential $\tilde{d}$. We are now in the setting of the perturbation lemma as in \cite[Lemma $2.1_*$]{HUKA}. The first part of the perturbation lemma is proved in \cite{brown}. The proof in \cite{brown} goes as follows: Define a map 
\[
H_t := \sum_{n\geq 0}(-Ht)^n H.
\]
It is not immediate that this is well-defined. To see this recall that there is a second grading on the space $T^c(sA)$ which we called the \textit{weight}. A homogeneous element 
\[
x_1\otimes ...\otimes x_n \in sA^{\otimes n}
\]
has degree $|x_1| + ... +|x_n|$ and weight $n$. Now notice that the map $t$ reduces the weight by $1$ and $H$ does not affect the weight. Thus for any element $x\in T^c(sA)$ there exists an $n$ such that $t_m (x) =0 $ for all $m\geq n$, and hence $H_t$ is well-defined. The following is immediate
\begin{equation}\label{temp}
	 H^2 = 0,\ PH = 0,\ HI=0 \Rightarrow H_t^2 = 0,\ PH_t = 0,\ H_t I=0.
\end{equation}
The following Lemma is also stated in \cite{brown}.
\begin{lemma}
\[
H_t D H =HDH_t = H,\ H_tD H_t = H_t
\]
\end{lemma}
\begin{proof}
    Let us prove $H_t D H =H$. $HDH_t = H$ is proved in the same way. 
	\begin{align*}
	H_t D H &= \sum_{n\geq 0} (-Ht)^nH (t+\tilde{d})H\\
	&= \sum_{n\geq 0} (-Ht)^nH tH+\sum_{n\geq 0} (-Ht)^nH (\mathrm{id}-IP-H\tilde{d})\\
	&\overset{\eqref{temp}}{=} -\sum_{n\geq 0} (-Ht)^{n+1} H +\sum_{n\geq 0} (-Ht)^nH\\
	&= H.
	\end{align*}
	The last equation follows:
	\[
	H_t D H_t = \sum_{n\geq 0}(-Ht)^n\underbrace{HDH_t}_{= H} =\sum_{n\geq 0}(-Ht)^n H = H_t.
	\]
\end{proof}
Define a preliminary map $\mathcal{D}_t := \mathrm{id} -DH_t - H_t D$. Note that $\mathcal{D}_t$ is idempotent, i.e.\ $\mathcal{D}_t ^2 = \mathcal{D}_t$. Indeed,
\begin{align}  \label{eq:idempotency_of_D_t}
	\mathcal{D}_t^2 &= (\mathrm{id} -DH_t - H_t D)(\mathrm{id} -DH_t - H_t D)\\
	&= \mathrm{id} -DH_t - H_t D -DH_t + D\underbrace{H_t DH_t}_{= H_t} -H_t D + \underbrace{H_t DH_t}_{= H_t} D\\
	&= \mathrm{id}-DH_t -H_t D = \mathcal{D}_t.
\end{align}
Now we make the following defintions $I_t := \mathcal{D}_t I$, $P_t := P\mathcal{D}_t$, and $d_\infty := P_tDI_t.$
Note that we may also write
\begin{align}
	I_t &= (\mathrm{id}-H_t D) I \label{eq:I_t}\\
	P_t &= P (\mathrm{id} - DH_t) \label{eq:P_t}\\
	d_\infty &= PDI_t = P_tDI. \label{eq:d_infty}
\end{align}

\begin{theorem} \textnormal{\cite{brown}}
The following is a deformation retract of chain complexes 
\[
\begin{tikzcd}
		\arrow[loop left, "H_t"] (T^c(sA),D) \arrow[r, shift left, "P_t"] & (T^c(sH_*(A)),d_\infty) \arrow[l, shift left, "I_t"].
	\end{tikzcd}
\]
\end{theorem}

\begin{proof}
Note that $D\mathcal{D}_t = \mathcal{D}_t D$ and thus
\[
d_\infty^2 \overset{\eqref{eq:d_infty}}{=} PDI_tP_tDI = PD\mathcal{D}_t DI = P\mathcal{D}_t D^2 I = 0.
\]
Hence $(T^c(sH_*(A)),d_\infty)$ is indeed a chain complex. We must to prove the following equations:
\begin{align}
	P_t I_t &= \mathrm{id} \label{1}\\
	\mathrm{id}- I_tP_t &= D H_t + H_tD \label{2}\\
	d_\infty P_t &= P_t D \label{4}\\
	D I_t &= I_t d_\infty \label{5}.
	\end{align}
Proof of equation (\ref{1}):
\begin{align*}
	P_t I_t &= P\mathcal{D}_t\mathcal{D}_t I \\
	&\overset{\eqref{eq:idempotency_of_D_t}}{=} P(\mathrm{id} -H_t D- DH_t)I \\
	&= PI = \mathrm{id}.
\end{align*}
Proof of equation \eqref{2}:
\begin{align*}
	I_tP_t &= (\mathrm{id}-H_t D) I P(\mathrm{id} - DH_t)\\
	&= (\mathrm{id}-H_t D) (\mathrm{id}-\tilde{d}H-H\tilde{d})(\mathrm{id} - DH_t)\\
	&= (\mathrm{id}-\tilde{d}H-H\tilde{d}-H_t D+H_tt\tilde{d}H+ \underbrace{H_t DH}_{=H}D)(\mathrm{id} - DH_t)\\
	&=(\mathrm{id}-\tilde{d}H-H_t D+H_ttdH )(\mathrm{id} - DH_t)\\
	&= \mathrm{id}-\tilde{d}H-H_t D +H_tt\tilde{d}H- DH_t + \tilde{d}\underbrace{HDH_t}_{=H}-H_tt\tilde{d}\underbrace{HDH_t}_{=H}\\
	&= \mathrm{id}-H_t D - DH_t.
\end{align*} 
Proof of equation (\ref{4}) and (\ref{5}):
\begin{align*}
	d_\infty P_t &\overset{\eqref{eq:d_infty}}{=} P DI_tP_t = P D\mathcal{D}_t = P \mathcal{D}_t D = P_t D\\
	I_t d_\infty &\overset{\eqref{eq:d_infty}}{=} I_t P_t DI = \mathcal{D}_tDI = D\mathcal{D}_t I = DI_t.
\end{align*}	
\end{proof}

The next proposition shows that $I_t$, $P_t$ and $d_\infty$ are compatible with the coalgebra structures. This is an extension of the results obtained in \cite{brown}. The next proposition is also proved in \cite[Lemma $2.1_*$]{HUKA} in more generality. As we are in a special case of Lemma $2.1_*$ in \cite{HUKA} there is a shorter and more explicit proof:
\begin{prop} \label{prop:perturbation_lemma_coalgebra}
	The following is a strong deformation retract of dg-coalgebras
	\[
	\begin{tikzcd}
		\arrow[loop left, "H_t"] (T^c(sA),D) \arrow[r, shift left, "P_t"] & (T^c(sH_*(A)),d_\infty) \arrow[l, shift left, "I_t"].
	\end{tikzcd}
	\]
\end{prop}
\begin{proof}
	Let us start with a straight forward computation
	\begin{align*}
	(P_t\otimes P_t) \Delta &= ((P - PtH_t)\otimes (P-PtH_t))\Delta	 \\
	&=\big(( P\sum_{i\geq 0} (-tH)^i) \otimes (P\sum_{j\geq 0} (-tH)^j)\big)\Delta	\\
	&= \sum_{i,j\geq 0} \big(P(-tH)^{i} \otimes P(-tH)^{j}\big)\Delta \\
	&= \sum_{n\geq 0} \sum_{\substack{i+j = n\\i,j\geq 0}} P(-tH)^i \otimes P (-tH)^{j}\big)\Delta\\
	&\overset{!}{=} \Delta P_t \\
	&= \Delta (P - PtH_t)\\
	&= (P\otimes P)\Delta - (P\otimes Pt + Pt\otimes P) \Delta H_t\\
	&= (P\otimes P)\Delta - \sum_{n\geq 0}(P\otimes Pt + Pt\otimes P) \Delta (-Ht)^nH.
	\end{align*}
	We compare the terms on both sides of the equation with the same number of $t$'s\footnote{This makes sense because these are the terms with the same weight.}. The terms with zero or only one $t$ cancel easy enough. For all $n\geq 2$ we get
	\begin{equation} \label{eq:induction}
		\sum_{\substack{i+j = n\\i,j\geq 0}} P(-tH)^i \otimes P (-tH)^{j}\big)\Delta \overset{!}{=} -(P\otimes Pt + Pt\otimes P)\Delta (-Ht)^{n-1} H,
	\end{equation}
	which we prove by induction. The induction start is $n=2$. Starting from the right-hand side of equation \eqref{eq:induction}:
	\begin{align*}
		(P\otimes Pt&+Pt\otimes P)\Delta HtH = (P\otimes Pt+Pt\otimes P)(\mathrm{id}\otimes H +H\otimes IP)\Delta tH\\
		&= (P\otimes PtH + PtH\otimes PIP)\Delta tH	\\
		&= (P\otimes PtH + PtH\otimes P)(t\otimes \mathrm{id} +\mathrm{id}\otimes t)\Delta H	\\
		&= (Pt\otimes PtH + P\otimes PtHt + PtHt\otimes P + PtH \otimes Pt)\Delta H	\\
		&= (Pt\otimes PtH + P\otimes PtHt + PtHt\otimes P + PtH \otimes Pt)\\
		&\quad \circ (H\otimes IP +\mathrm{id}\otimes H)\Delta 	\\
		&= (P\otimes PtHtH + PtHtH\otimes \underbrace{P IP}_{= P} + PtH\otimes PtH)\Delta
	\end{align*}
	which is precisely equal to the left-hand side of equation \eqref{eq:induction}. Now assume equation \eqref{eq:induction} holds for a fixed $n\geq 2$ and compute the following
	\begin{align*}
	    &-(P\otimes Pt + Pt\otimes P)\Delta (-Ht)^{n} H \\
	    &\overset{\mathrm{i.a.}}{=}- \big(\sum_{i+j=n} P(-tH)^{i} \otimes P(-tH)^j \big)\Delta tH\\
	    &=- \big(\sum_{i+j=n} P(-tH)^{i} \otimes P(-tH)^j \big) (t\otimes H + tH\otimes IP + \mathrm{id}\otimes tH - H\otimes tIP)\Delta \\
	    &= -\big(P(-tH)^n tH \otimes PIP +\sum_{i+j=n} P(-tH)^{i}\otimes P(-tH)^{j}tH\big)\Delta \\
	    &= \sum_{i+j=n+1} (P(-tH)^{i}\otimes P(-tH)^{j})\Delta.
	\end{align*}
	This completes the induction step. Thus $P_t$ is a coalgebra map. The computation showing that $I_t$ is a coalgebra map is similar and is omitted here. From $d_\infty = P_tDI_t$, it follows that $d_\infty$ is a coderivation along $P_tI_t=\mathrm{id}$.
\end{proof}

\subsection{Perturbation Lemma and \texorpdfstring{$A_\infty$}{TEXT}-algebras}
We now translate the results back into the $A_\infty$-algebra setting. Proposition \ref{prop:(A,m)==d^2} implies that $d_\infty$ induces an $A_\infty$-structure on the homology $(H_*(A),m)$. Corollary \ref{cor:A_infty_morphisms=F} implies that $I_t$, $P_t$ induce $A_\infty$-morphisms $I_\infty$, $P_\infty$, respectively. We can prove Theorem \ref{thm:minimal_model} in the following special case.

\begin{theorem}\label{thm:existance_minimal_model}
	Let $(A,d)$ be a dg-algebra then there exists a minimal model
	\[
	i: (H_*(A),m) \rightarrow (A,d).
	\]
	Moreover, the $A_\infty$-structure on $H_*(A)$ is unique up to isomorphism of $A_\infty$-algebras.
\end{theorem}
\begin{proof}
	We can rewrite the differential $d_\infty$ in a more explicit form:
	\begin{align} \label{eq:A_infty_structure_from_perturbation_lemma}
		d_\infty &= P_t DI_t = P\mathcal{D}_tD\mathcal{D}_tI=PD\mathcal{D}_t\mathcal{D}_tI \nonumber\\
		& = PD\mathcal{D}_tI= Pt(\mathrm{id} - H_tt)I\nonumber \\
		&=PtI - \sum_{n\geq 0}Pt(-Ht)^nHtI\nonumber \\
		&= \sum_{n\geq 0}Pt(-Ht)^nI = \sum_{n\geq 0}P(-tH)^ntI.
	\end{align}
	We first show that $ m_1 = 0$. For this simply note that $I$ is equal to $sis^{-1}$ on $sH^{*}(A)$ and $t$ is equal to zero on $sA$. Secondly, we show that the $m_2$ map is induced my $\nu$, the product on $A$. Indeed, $I$ restricted to $sH^{*}(A)\otimes sH^{*}(A)$ is equal to $sis^{-1}\otimes sis^{-1}$, $t$ is equal to $s\nu(s^{-1})^{\otimes 2}$ on $sA\otimes sA$, and $P$ is equal to $sps^{-1}$ on $sA$, hence $ -s^{-1}( d_\infty )_2 s^{\otimes 2} $ is equal to
	\[
	p\circ \nu \circ (i\otimes i).
	\]
    To see that this is equal to the induced product, let $[x], [y] \in H_*(A)$ be arbitrary elements we must show that
    \[
    [x\wedge y] = p(i[x]\wedge i[y]).
    \]
    Recall that $pi = \mathrm{id}$ thus it is equivalent to showing that
    \[
    p(i[x\wedge y] - i[x] \wedge i[y]) = 0.
    \]
    From $i$ being a cycle choosing map we see that $i[x\wedge y] - i[x] \wedge i[y]$ is always an exact form. Now recall that by construction of the projection map p, we know that $p$ is zero on exact forms. Lastly, let us rewrite $I_t$ more explicitly as 
	\[
	I_t = \sum_{n\geq 0} (-Ht)^n I.
	\]
	From this it immediately follows that $(I_\infty)_1 = i$. This proves existence. Recall that there also exists a map. $P_\infty:(A,d)\rightarrow (H_*(A),m)$ with $(P_\infty)_1 = p$. Let 
	\[
	I_\infty':(H_*(A),m')\rightarrow (A,d)
	\]
	be another minimal model. Then 
	\[
	P_\infty \circ I_\infty' :(H_*(A),m')\rightarrow (H_*(A),m)
	\]
	is an isomorphism of $A_\infty$-algebras as $(P_\infty \circ I_\infty')_1 = p\circ (I_\infty')_1= \mathrm{id}$. This proves uniqueness.
\end{proof}
For a quasi-isomorphism $f$ of dg-algebras there need not exist a morphism $f^{-1}$ of dg-algebras such that $H_*(f^{-1}) = H_*(f)^{-1}$. The following Lemma shows that allowing $A_\infty$-morphisms is a way to get around this. 
\begin{lemma} \label{lem:quasi-isomorphic-inverse}
\textnormal{\cite[Theorem 10.4.4.]{loday_vallette}} Assume $f:(A,d)\rightarrow (A',d')$ is a quasi-isomorphism of dg-algebras, then there exists a quasi-isomorphism of $A_\infty$-algebras
\[
g: (A',d')\rightarrow (A,d),
\]
such that $H_*(g_1) = H_*(f)^{-1}$.
\end{lemma}
\begin{proof}
	Let 
	\[
	\begin{tikzcd}
		(A,d) \arrow[r, shift left, "P_\infty"] & (H_*(A),m) \arrow[l, shift left, "I_\infty"]
	\end{tikzcd}
	\]
	and 
	\[
	\begin{tikzcd}
		(A',d') \arrow[r, shift left, "P_\infty'"] & (H_*(A'),m') \arrow[l, shift left, "I_\infty'"]
	\end{tikzcd}
	\]
	denote $A_\infty$-quasi-isomorphisms coming from deformation retracts. Consider the composition
	\[
	\begin{tikzcd}
		(H_*(A),m) \arrow[r, "I_\infty"] & (A,d) \arrow[r, "f"] &(A',d') \arrow[r, "P_\infty'"]& (H_*(A'),m')
	\end{tikzcd}
	\]
	and note that $(P_\infty' \circ f \circ I_\infty)_1 = p'\circ f\circ i = H_*(f)$, which is invertible. Thus by Proposition \ref{prop:inverse} the $A_\infty$-morphism $P_\infty' \circ f \circ I_\infty$ is invertible. The claimed quasi-isomorphism is now simply the composition
	\[
	\begin{tikzcd}
	    (A',d') \arrow[r,"P_\infty'"] & (H_*(A'),m') \arrow[rr, "(P_\infty' \circ f \circ I_\infty)^{-1}"] && (H_*(A),m) \arrow[r, "I_\infty"] & (A,d).
	\end{tikzcd}
	\]
\end{proof}
\begin{remark} \label{rmk:1->2}
This allows us to prove part of 1.\ $\rightarrow$ 2.\ of Theorem \ref{thm:formality_def_equivalenz}, i.e.\ it remains to prove that the resulting map is balanced. If 
\[
    (H_*(A),0)\leftarrow \bigcdot \ ... \ \bigcdot \rightarrow (A,d)
\]
is a zig-zag of quasi-isomorphisms then simply ``invert'' all arrows pointing to the left, using Lemma \ref{lem:quasi-isomorphic-inverse}, and then compose to get a quasi-isomorphism of $A_\infty$-algebras
\[
(H_*(A),0)\rightarrow (A,d).
\]
\end{remark}

\section{Explicit Proofs}
\label{sec:explicit_proofs}
As in the previous section we fix a dg-algebra $(A,d,\nu)$ and a deformation retract
\[	
\begin{tikzcd}
		\arrow[loop left, "h"] (A,d) \arrow[r, shift left, "p"] & \arrow[l, shift left, "i"](H_*(A),0).
\end{tikzcd}
\]
\subsection{Explicit Minimal Model}
The aim of this subsection is to get an explicit formula for the $A_\infty$-structure on the homology as constructed in the previous section and compare the formula to the one in \cite{merkulov}. Notice that $s^{\otimes n} (s^{-1})^{\otimes n}:(sA)^{\otimes n}\rightarrow (sA)^{\otimes n}$ is not simply the identity but comes with a sign, due to the Koszul sign convention. Let $\alpha_1 = 1$ and
\[
\alpha_n := (-1)^{\sum_{i =1}^{n-1} i} = (-1)^{\frac{n(n-1)}{2}} \ \mathrm{for}\ n\geq 2.
\]
The following equations are immediate for all $n\geq 1$
\[
\alpha_{n} (-1)^n = \alpha_{n+1}, \quad \alpha_{n+2} = 	-\alpha_n.
\]
By induction we get for $n\geq 1$ 
\[
s^{\otimes n} (s^{-1})^{\otimes n} = \alpha_n \mathrm{id}^{\otimes n}\ .
\]
Let us set
\[
h_n := \sum_{r+t = n-1} \mathrm{id}^{\otimes r}\otimes h\otimes (ip)^{\otimes t} , \quad \mathrm{for}\ n\geq 1
\]
and 
\[
\nu_n := \sum_{r+t = n-2} (-1)^t \mathrm{id}^{\otimes r}\otimes \nu \otimes \mathrm{id}^{\otimes t}, \quad \mathrm{for}\ n\geq 2.
\]
Let $\chi_2:= \nu$ and for all $n\geq 3$
\[
\chi_{n} := (-1)^{n-1}\chi_{n-1} h_{n-1}\nu_{n}.
\]
A short induction argument shows that the $A_\infty$-structure on $H_*(A)$ from equation \eqref{eq:A_infty_structure_from_perturbation_lemma} can be written for all $n\geq 2$
\[
m_n = p\chi_n i^{\otimes n}.
\]
Before embarking on the proof, relating this to Merkulov's minimal model, notice that for $n\geq 1$
\begin{align*}
	h_n \nu_{n+1} i^{\otimes n+1} &= \sum_{r+t=n-1} (-1)^{t}(\mathrm{id}^{\otimes r}\otimes h\nu\otimes (ip)^{\otimes t} )i^{\otimes n+1}
\end{align*}
due to $h i= 0$ and $pi = \mathrm{id}$. More generally, let $f_1,...,f_{n+1}$ be placeholder maps either equal to $i:H_*(A)\rightarrow A$ or equal to $h\circ f:B \rightarrow A$, where $B$ is an arbitrary graded space and $f :B \rightarrow A$ is an arbitrary graded linear map. Let us denote this as 
\[
f_1,...,f_{n+1} \in \{i, hf: f\ \mathrm{arbitrary}\}.
\]
It is easy to see that
\begin{align}\label{eq:alpha}
    &h_n \nu_{n+1} (f_1\otimes ...\otimes f_{n+1})= \sum_{r+t=n-1}(-1)^t f_1\otimes ... \otimes f_r\nonumber \\
    &\qquad\qquad\qquad\qquad \otimes h\nu(f_{r+1}\otimes f_{r+2})\otimes ip f_{r+3}\otimes ...\otimes ip f_{n+1}.
\end{align}
Moreover, one can show that for all $n\geq 2$ and $k,l\geq 1$ such that $k+l = n$
\begin{equation} \label{eq:h_nu}
	h_{n}\nu_{n+1} (f_1 \otimes ... \otimes f_{n+1}) = ((-1)^l h_k\nu_{k+1} \otimes (ip)^{\otimes l} + \mathrm{id}^{\otimes k}\otimes h_l\nu_{l+1})(f_1 \otimes ... \otimes f_{n+1}). 
\end{equation}
Now we are set up to prove the main result, which helps us relate the two $A_\infty$-structures. The formula in the next proposition can be guessed after doing the first cases in $n$. A similar formula is given in \cite{merkulov}.
\begin{prop}
	For all $n\geq 3$ and all $f_1,...,f_n\in \{i,hf:f \ \mathrm{arbitary}\}$
	\begin{align}\label{eq:chi_m}
		 &\chi_n \circ (f_1\otimes... \otimes f_n) = \big[ \sum_{\substack{k+l=n\\k,l\geq 2}}(-1)^{k(l+1)}\nu(h\chi_k \otimes h\chi_l (ip)^{\otimes l}) \nonumber \\
		 &\qquad\qquad+(-1)^{(n-1)}\nu(\mathrm{id}\otimes h\chi_{n-1})-\nu(h\chi_{n-1}\otimes ip)\big]\circ (f_1\otimes... \otimes f_n).
	\end{align}
\end{prop}
\begin{proof} The proof is by induction. For $n=3$ and $f_1, f_{2}, f_{3}\in \{i,hf:f \ \mathrm{arbitary}\}$, equation \eqref{eq:chi_m} is just a special case of equation \eqref{eq:h_nu}:
	\begin{align*}
	    \chi_3(f_1\otimes f_{2}\otimes f_{3}) &= \nu h_2\nu_3(f_1\otimes f_{2}\otimes f_{3})\\
	    &= \nu(-h\nu \otimes ip + \mathrm{id}\otimes h\nu)(f_1\otimes f_{2}\otimes f_{3})\\
	    &= \big[-\nu(h\chi_2\otimes \mathrm{id}) + \nu(\mathrm{id}\otimes h\chi_2)\big](f_1\otimes f_{2}\otimes f_{3}).
	\end{align*}
	Assume the result holds for a fixed $n\geq 3$. Let $f_1,...,f_{n+1}\in \{i,hf:f \ \mathrm{arbitary}\}$ we compute
	\begin{equation*}
		\chi_{n+1}(f_1\otimes...\otimes f_{n+1}) = (-1)^n\chi_n h_n\nu_{n+1} (f_1\otimes...\otimes f_{n+1})
	\end{equation*}
	by equation \eqref{eq:alpha}, we can imply the induction assumption, at the same time we use equation \eqref{eq:h_nu} to rewrite $h_{n}\nu_{n+1}$
	\begin{align*}
		&= (-1)^n\big[\sum_{\substack{k+l=n\\k,l\geq 2}}(-1)^{k(l+1)} \nu(h\chi_k \otimes h\chi_l(ip)^{\otimes l})((-1)^l h_k\nu_{k+1} \otimes (ip)^{\otimes l} + \mathrm{id}^{\otimes k}\otimes h_l\nu_{l+1})\\
		&\ +(-1)^{n-1} \nu(\mathrm{id}\otimes h\chi_{n-1})((-1)^{n-1}h\nu \otimes (ip)^{\otimes n-1} + \mathrm{id}\otimes h_{n-1}\nu_{n})\\
		&\ -\nu(h\chi_{n-1}\otimes ip)(-h_{n-1}\nu_n\otimes ip + \mathrm{id}^{\otimes n-1}\otimes h\nu)\big](f_1\otimes...\otimes f_{n+1})\\
		&=\big[ \sum_{\substack{k+l=n\\k,l\geq 2}}(-1)^{k+l}(-1)^{k(l+1)}(-1)\nu(h\underbrace{\chi_k h_k\nu_{k+1}}_{=(-1)^k\chi_{k+1}}\otimes h\chi_l(ip)^{\otimes l})\\
		&\ +\nu (h\nu \otimes h\chi_{n-1}(ip)^{\otimes n-1})-\nu(\mathrm{id}\otimes h\underbrace{\chi_{n-1}h_{n-1}\nu_n}_{=(-1)^{n-1}\chi_{n}})\\
		&\ +(-1)^n\nu(h\underbrace{\chi_{n-1}h_{n-1}\nu_{n}}_{=(-1)^{n-1}\chi_{n}}\otimes ip)\big](f_1\otimes...\otimes f_{n+1})
	\end{align*}	
    now reorder the sum with $k' := k+1$ and $l':=l$
	\begin{align*}
		&=\big[\sum_{\substack{k'+l'=n+1\\3\leq k'\leq n-1\\2\leq l'\leq n-2}}(-1)^{k'(l'+1)}\nu(h\chi_{k+1}\otimes h\chi_l(ip)^{\otimes l}) + \nu (h\chi_2 \otimes h\chi_{n-1}(ip)^{\otimes n-1})\\
		&\ +(-1)^{n}\nu(\mathrm{id}\otimes h\chi_{n}) - \nu(h\chi_{n}\otimes ip)\big](f_1\otimes...\otimes f_{n+1}).
	\end{align*}
\end{proof}	

The use of the placeholder function a trick to be able to imply the induction assumption right away. The only case of equation \eqref{eq:chi_m} we are interested in is $f_1 = ...= f_{n} =i$, in this case equation \eqref{eq:chi_m} simplifies as
\begin{align} \label{eq:merkulov_minimal}
	 \chi_n \circ i^{\otimes n} &= \big[ \sum_{\substack{k+l=n\\k,l\geq 2}}(-1)^{k(l+1)}\nu(h\chi_k \otimes h\chi_l )\nonumber \\
		 &\ +(-1)^{(n-1)}\nu(\mathrm{id}\otimes h\chi_{n-1})-(h\chi_{n-1}\otimes \mathrm{id})\big]\circ i^{\otimes n}.
\end{align}
This formula allows an explicit proof of the $A_\infty$-relations, see \cite{merkulov}. We give a proof of Lemma 3.2 from \cite{merkulov} to convince the reader that the signs in equation \eqref{eq:merkulov_minimal} are correct. Formally set ``$\lambda_1:=-h^{-1}$'', $\lambda_2:= \nu$, and for $n\geq 3$ set
\begin{equation} \label{eq:lambda}
	\lambda_n := \sum_{\substack{k+l=n\\k,l\geq 1}} (-1)^{k(l+1)}\nu(h\lambda_k\otimes h\lambda_l).
\end{equation}
\begin{lemma}\textnormal{\cite{merkulov}} The following equation is true for all $n\geq 3$ 
\[
\Phi_n= \sum_{\substack{j+l+k=n\\2\leq l\leq n-1\\j,k\geq 0}} = (-1)^{jl+k} \lambda_{j+k+1}(\mathrm{id}^{\otimes j}\otimes \lambda_l\otimes \mathrm{id}^{\otimes k})=0.
\]
\end{lemma}
\begin{proof}
	The proof in \cite{merkulov} is an induction proof and goes as follows: Let $n\geq 4$ first split off the extreme values of $j,k$:
	\begin{align} \label{eq:temp_lem_3.2_PHI}
		\Phi_n &= \sum_{\substack{j+l=n\\j\geq 1,l\geq 2}}(-1)^{jl}\lambda_{j+1}(\mathrm{id}^{\otimes j}\otimes \lambda_l)+\sum_{\substack{l+k=n\\l\geq 2, k\geq 1}}(-1)^{k}\lambda_{k+1}( \lambda_l\otimes \mathrm{id}^{\otimes k})\nonumber\\
		&+ \sum_{\substack{j+l+k=n\\ l\geq 2, j,k\geq 1}} (-1)^{jl+k} \lambda_{j+k+1}(\mathrm{id}^{\otimes j}\otimes \lambda_l\otimes \mathrm{id}^{\otimes k}).
	\end{align}
	Expand the first and second sum with the definition of $\lambda_{j+1}, \lambda_{k+1}$, respectively. The first two sums become
	\begin{align*}
		&\ \sum_{\substack{j+l=n\\j\geq 1,l\geq 2}} (-1)^{jl}\sum_{\substack{s+t= j+1\\s,t\geq 1}} (-1)^{s(t+1)}\nu(h\lambda_s\otimes h\lambda_t)(\mathrm{id}^{j}\otimes \lambda_l)\\
		&\ +\sum_{\substack{l+k=n\\l\geq 2,k\geq 1}} (-1)^{k}\sum_{\substack{s+t= k+1\\s,t\geq 1}} (-1)^{s(t+1)}\nu(h\lambda_s\otimes h\lambda_t)(\lambda_l\otimes \mathrm{id}^{k})\\
		&=\sum_{\substack{j+l=n\\j\geq 1,l\geq 2}} (-1)^{jl}\sum_{\substack{s+t= j+1\\s,t\geq 1}} (-1)^{s(t+1)}\nu(h\lambda_s\otimes h\lambda_t (\mathrm{id}^{\otimes t-1}\otimes \lambda_l))\\
		&\ +\sum_{\substack{l+k=n\\l\geq 2,k\geq 1}} (-1)^{k}\sum_{\substack{s+t= k+1\\s,t\geq 1}} (-1)^{s(t+1)}\nu(h\lambda_s(\lambda_l\otimes \mathrm{id}^{\otimes s-1})\otimes h\lambda_t):= \spadesuit .
	\end{align*}
	Consider the case $s=1$ and $t=k$ for the second sum:
	\begin{align*}
		&-\sum_{\substack{l+k=n\\l\geq , k\geq 1}}(-1)^{k+(l+1)(k+1)} \nu(\lambda_l\otimes h\lambda_k)\\
		&= \sum_{\substack{l+k=n\\l\geq 2, k\geq 1}}(-1)^{l(k+1)}\sum_{\substack{s+t=l\\s,t\geq 1}} (-1)^{s(t+1)}\underbrace{\nu(\nu(h\lambda_s\otimes h\lambda_t)\otimes h\lambda_k)}_{=\nu(h\lambda_s \otimes \nu (h\lambda_t \otimes h\lambda_k))}\\
		&= \sum_{k=1}^{n-2} (-1)^{(n-k)(k+1)} \sum_{s=1}^{n-k}(-1)^{s(n-k-s+1)} \nu(h\lambda_s\otimes \nu(h\lambda_{n-k-s}\otimes h\lambda_k))\\
		&= \sum_{s=1}^{n-2} (-1)^{s(n-s)} \sum_{k=1}^{n-s-1} (-1)^{(n-k-s)(k+1)} \nu(h\lambda_s\otimes \nu(h\lambda_{n-k-s}\otimes h\lambda_k)) \\
		&= \sum_{s=1}^{n-2}(-1)^{s(n-s)} \sum_{k+l=n-s} (-1)^{l(k+1)} \nu(h\lambda_s\otimes \nu(h\lambda_{l}\otimes h\lambda_k))\\
		&= \sum_{\substack{s+t=n\\s\geq 1,t\geq 2}} (-1)^{st}\nu(h\lambda_s\otimes h\lambda_t).
 	\end{align*}
	Notice that this sum cancels with the case $s=j$ and $t=1$ of the first sum in $\spadesuit$. Thus 
	\begin{align*}
		\spadesuit &= \sum_{\substack{j+l=n\\j\geq 1,l\geq 2}} (-1)^{jl}\sum_{\substack{s+t= j+1\\s\geq 1, t\geq 2}} (-1)^{s(t+1)}\nu(h\lambda_s\otimes h\lambda_t (\mathrm{id}^{\otimes t-1}\otimes \lambda_l))\\
		&\ +\sum_{\substack{l+k=n\\l\geq 2,k\geq 1}} (-1)^{k}\sum_{\substack{s+t= k+1\\s\geq 2,t\geq 1}} (-1)^{s(t+1)}\nu(h\lambda_s(\lambda_l\otimes \mathrm{id}^{\otimes s-1})\otimes h\lambda_t)\\
		&= \sum_{l=2}^{n-1}(-1)^{l(n-l)}\sum_{s=1}^{n-l-1} (-1)^{s(n-s-l+2)}\nu(h\lambda_s\otimes h\lambda_{n-s-l+1}(\mathrm{id}^{\otimes n-s-l}\otimes \lambda_l))\\
		&\ +\sum_{l=2}^{n-1}(-1)^{n-l}\sum_{t=1}^{n-l-1}(-1)^{(n-t-l+1+l)(t+1)}\nu(h\lambda_{n-l-t+1}(\lambda_l\otimes \mathrm{id}^{\otimes n-l-t})\otimes h\lambda_t)\\
		&= \sum_{s=1}^{n-3}(-1)^{s(n-s)}\sum_{l=2}^{n-s-1}(-1)^{l(n-l)+ls} \nu(h\lambda_s\otimes h\lambda_{n-s-l+1}(\mathrm{id}^{\otimes n-s-l}\otimes \lambda_l))\\
		&\ +\sum_{t=1}^{n-3}\sum_{l=2}^{n-1-t} (-1)^{n+1+n-t+1t(n-t+)}\nu(h\lambda_{n-l-t+1}(\lambda_l\otimes \mathrm{id}^{\otimes n-l-t})\otimes h\lambda_t).
	\end{align*}
	In the first sum set $j:= n-l-s$ and in the second sum set $k:=n-l-t$:
	\begin{align} \label{eq:temp_lem_3.2}
		&= \sum_{s=1}^{n-3}(-1)^{s(n-s)}\sum_{\substack{j+l=n-s\\j\geq 1, l\geq 2}} (-1)^{lj} \nu(h\lambda_s\otimes h\lambda_{j+1}(\mathrm{id}^{\otimes j}\otimes \lambda_l))\nonumber \\
		&\ +\sum_{t=1}^{n-3}(-1)^{t(n-1+1)+n+1} \sum_{\substack{l+k=n-t\\l\geq 2, k\geq 1}} (-1)^k \nu(h\lambda_{k+1}(\lambda_l\otimes \mathrm{id}^{\otimes k})\otimes h\lambda_t) .
	\end{align}
	Now turn back to equation \eqref{eq:temp_lem_3.2_PHI} and compute the sum that contains no extreme values of $j$ and $k$:
	\begin{align*}
		&\ \sum_{\substack{j+l+k=n\\ l\geq 2, j,k\geq 1}} (-1)^{jl+k} \lambda_{j+k+1}(\mathrm{id}^{\otimes j}\otimes \lambda_l\otimes \mathrm{id}^{\otimes k})\\
		&= \sum_{\substack{j+l+k=n\\ l\geq 2, j,k\geq 1}} (-1)^{jl+k} \sum_{\substack{s+t=j+k+1\\s,t\geq 1}} (-1)^{s(t+1)}\nu(h\lambda_s\otimes h\lambda_t) (\mathrm{id}^{\otimes j}\otimes \lambda_l\otimes \mathrm{id}^{\otimes k}).
	\end{align*}
	Split the second sum into $1\leq s\leq j$ and $j+1\leq s\leq j+k\Leftrightarrow 1\leq t\leq k$:
	\begin{align*}
		&= \sum_{l=2}^{n-2}\sum_{k=1}^{n-1-l}(-1)^{l(n-l-k)}\sum_{s=1}^{n-k-l}(-1)^{s(n-l-s+2)}\\
		&\qquad\qquad\qquad\nu(h\lambda_s\otimes h\lambda_{n-l-s+1}(\mathrm{id}^{\otimes n-k-l-s}\otimes \lambda_l \otimes \mathrm{id}^{\otimes k})\\
		&\ +\sum_{l=2}^{n-2}\sum_{j=1}^{n-1-l}(-1)^{jl+n-j-l}\sum_{t=1}^{n-l-j}(-1)^{(n-l-t+1)(t+1)}(-1)^{l(t+1)}\\
		&\qquad\qquad\qquad \nu(h\lambda_{n-l-t+1}(\mathrm{id}^{\otimes j}\otimes \lambda_l\otimes \mathrm{id}^{\otimes n-l-j-t})\otimes h\lambda_t).
	\end{align*}
	The sign $(-1)^{l(t+1)}$ is due to the Koszul sign convention. Next permute the sums:
	\begin{align*}
		&= \sum_{s=1}^{n-3}(-1)^{s(n-s)}\sum_{l=2}^{n-1-s}\sum_{k=1}^{n-s-l} (-1)^{l(n-l-k-s)+k}\\
		&\qquad\qquad\qquad\nu(h\lambda_s\otimes h\lambda_{n-l-s+1}(\mathrm{id}^{\otimes n-k-l-s}\otimes \lambda_l \otimes \mathrm{id}^{\otimes k})\\
		&\ +\sum_{t=1}^{n-3}(-1)^{t(n-t+1)+n+1}\sum_{l=2}^{n-1-t}\sum_{j=1}^{n-t-l}(-1)^{lj+n-j-l-t} \\
		&\qquad\qquad\qquad \nu(h\lambda_{n-l-t+1}(\mathrm{id}^{\otimes j}\otimes \lambda_l\otimes \mathrm{id}^{\otimes n-l-j-t})\otimes h\lambda_t).
	\end{align*}
	Set $j':= n-k-l-s$ in the first sum and $k':=n-l-j-t$ in the second sum:
	\begin{align*}
		&= \sum_{s=1}^{n-3}(-1)^{s(n-s)}\sum_{\substack{j'+l+k=n-s\\ j'\geq 0, l\geq 2,k\geq 1}} (-1)^{l j'+k} \nu(h\lambda_s\otimes h\lambda_{j'+t+1}(\mathrm{id}^{\otimes j'}\otimes \lambda_l \otimes \mathrm{id}^{\otimes k})\\
		&\ +\sum_{t=1}^{n-3}(-1)^{t(n-t+1)+n+1}\sum_{\substack{j+l+k'=n-t\\j\geq 1, l\geq 2,k\geq 0}}(-1)^{jl+k'} \nu(h\lambda_{j+k'+1}(\mathrm{id}^{\otimes j}\otimes \lambda_l\otimes \mathrm{id}^{\otimes k'})\otimes h\lambda_t).
	\end{align*}
	In order to simplify this, using the definition of $\Phi_n$, notice that the first sum is missing the case $k=0$ and the second sum the case $j=0$ these two cases are provided by equation \eqref{eq:temp_lem_3.2}. Putting everything back into equation \eqref{eq:temp_lem_3.2_PHI} we get
	\begin{equation*}
		\Phi_n = \sum_{s=1}^{n-3} (-1)^{s(n-s)} \nu(h\lambda_s\otimes \Phi_{n-s})+ \sum_{t=1}^{n-3}(-1)^{t(n-t+1)+n+1}\nu(\Phi_{n-t}\otimes h\lambda_t)
	\end{equation*}
	for all $n\geq 4$. Now simply observe that for $n=3$
	\[
	\Phi_ 3 = \nu(\nu\otimes \mathrm{id}- \mathrm{id}\otimes \nu)=0.
	\]
	Thus by the above inductive formula $\Phi_n=0$ for all $n\geq 3$. 
\end{proof}

This constitutes the main part of the proof of the following Theorem in \cite{merkulov}.
\begin{theorem} \label{thm:merkulov} \textnormal{\cite{merkulov}}
Let $(A,d,\nu)$ be a dg-algebra and
\[
\begin{tikzcd}
		\arrow[loop left, "h"] (A,d) \arrow[r, shift left, "p"] & \arrow[l, shift left, "i"](H_*(A),0)
\end{tikzcd}
\]
a choice of deformation retract. Then $m_1 := 0$ and $m_n:= p\circ \lambda_n \circ i^{\otimes n}$ for $n\geq 2$ defines an $A_\infty$-structure on $H_*(A)$. Moreover, $i_1:=i$ and $i_n:= \lambda_n \circ i^{\otimes n}$ for $n\geq 2$ defines a minimal model
\[
i:(H_*(A),m)\rightarrow (A,d).
\]
\end{theorem}

\subsection{The Projection Map I}
The purpose of this subsection is to give a direct proof of the fact that there exists an $A_\infty$-morphism in the opposite direction
\[
p:(A,d)\rightarrow (H_*(A),m)
\]
such that the first component $p_1$ is simply the projection map $p: (A,d)\rightarrow (H_*(A),0)$.

\vspace{4mm}

We set up some notation: Let $d_1 := d$ and for $n\geq 2$, $d_n:A^{\otimes n} \rightarrow A^{\otimes n}$ is given by
\[
d_n:= (-1)^{n-1} \sum_{r+1+t=n} \mathrm{id}^{\otimes r}\otimes d\otimes \mathrm{id}^{\otimes t}.
\]
It is immediate that for $n\geq 2$ and $i,j\geq 1$ such that $i+j=n$ we have
\begin{equation} \label{eq:d_n_split}
	d_n = (-1)^{i}\mathrm{id}^{\otimes i}\otimes d_j +(-1)^{j}d_i\otimes \mathrm{id}^{\otimes j}.
\end{equation}
Before getting started on the projection map let us show some properties of the $h_n$, $\nu_n$, and $d_n$ maps.
\begin{lemma} \label{lem:nu_nd_n=d_nnu_n}
	For all $n\geq 2$
	\begin{equation} \label{eq:nu-d=d-nu}
		\nu_n d_{n}+d_{n-1} \nu_n =0 .
	\end{equation}
\end{lemma}
\begin{proof}
	The case $n =2$ is true because $d$ is a derivation. Assume this holds for a fixed $n\geq 2$ then
	\begin{align*}
		\nu_{n+1} d_{n+1} &=( -\nu_n\otimes \mathrm{id} +\mathrm{id}^{\otimes n-1}\otimes \nu)d_{n+1}\\
		&\overset{\eqref{eq:d_n_split}}{=} -\nu_n\otimes \mathrm{id} ((-1)^n \mathrm{id}^{\otimes n}\otimes d -d_n \otimes \mathrm{id})\\
		&\ + \mathrm{id}^{\otimes n-1}\otimes \nu ( (-1)^{n-1}\mathrm{id}^{\otimes n-1} \otimes d_2 +d_{n-1}\otimes \mathrm{id}^{\otimes 2})\\
		&= (-1)^{n-1}\nu_n \otimes d + \nu_n d_n\otimes \mathrm{id} +(-1)^{n-1} \mathrm{id}^{\otimes n-1}\otimes \nu d_2 +d_{n-1} \otimes \nu\\
		&\overset{\mathrm{i.a.}}{=} (-1)^{n-1}\nu_n \otimes d - d_{n-1}\nu_n\otimes \mathrm{id} +(-1)^{n} \mathrm{id}^{\otimes n-1}\otimes d\nu +d_{n-1} \otimes \nu \\
		&= ((-1)^{n-1}\mathrm{id}^{\otimes n-1}\otimes d - d_{n-1}\otimes \mathrm{id})\underbrace{(\nu_n\otimes \mathrm{id}- \mathrm{id}^{\otimes n-1}\otimes \nu)}_{=-\nu_{n+1}}\\
		&= - d_n\nu_{n+1}.
	\end{align*}
\end{proof}
Equation \eqref{eq:deformation_retract_dgca} tells us that
\begin{equation}\label{eq:id-ip=d_nh_n}
	\mathrm{id}^{\otimes n} - (ip)^{\otimes n} = (-1)^{n-1} (h_n d_n + d_n h_n).
\end{equation}
Let us give a proof of equation \eqref{eq:id-ip=d_nh_n} that does not use any of the results from section \ref{sec:pertub_lemma}; similar to the proof of Lemma \ref{lem:nu_nd_n=d_nnu_n} by induction on $n$. For $n = 1$, equation \eqref{eq:id-ip=d_nh_n} is the property of the deformation retract. Assume equation \eqref{eq:id-ip=d_nh_n} holds for a fixed $n\geq 1$ and compute
\begin{align*}
		&\ (-1)^nh_{n+1}d_{n+1} \\
		&= (-1)^n(h_n\otimes ip + \mathrm{id}\otimes h)(-d_n \otimes \mathrm{id}+(-1)^n \mathrm{id}^{\otimes n} \otimes d)\\
		&= (-1)^{n-1} h_nd_n\otimes ip +(-1)^n d_n \otimes h + \mathrm{id}^{\otimes n}\otimes hd\\
		&\overset{\mathrm{i.a.}}{=} ((ip)^{\otimes n} -\mathrm{id}^{\otimes n} + (-1)^n d_nh_n)\otimes ip +(-1)^n d_n \otimes h + \mathrm{id}^{\otimes n}\otimes (\mathrm{id}- ip-dh)\\
		&=\mathrm{id}^{\otimes n+1} - (ip)^{\otimes n+1} +(-1)^nd_nh_n\otimes ip +(-1)^{n} d_n \otimes h - \mathrm{id}^{\otimes n}\otimes dh - \underbrace{h_n\otimes dip}_{=0}\\
		&= \mathrm{id}^{\otimes n+1} - (ip)^{\otimes n+1} -(-1)^n \underbrace{(-d_n\otimes \mathrm{id} +(-1)^n \mathrm{id}^{\otimes n} \otimes d)}_{= d_{n+1}} (h_n \otimes ip + \mathrm{id}^{\otimes n} \otimes h)\\
		&= \mathrm{id}^{\otimes n+1} - (ip)^{\otimes n+1} - (-1)^n d_{n+1}h_{n+1}.
\end{align*}
The projection map from Section \ref{sec:pertub_lemma}
\[
P_t = \sum_{n\geq 0} P (-tH)^n
\]
gives us an $A_\infty$-morphism $p:(A,d)\rightarrow (H_*(A),m)$ with $p_1 = p$ and for $n\geq 2$
\begin{equation} \label{eq:p_n_definition}
    p_n = \alpha_{n-1} p \nu_2 h_2... \nu_n h_n.
\end{equation}
From which we get the following recursive formula: $p_1 = p$ and for $n\geq 2$
\[
p_n = (-1)^n p_{n-1}\nu_n h_n.
\]
Recall that $\chi_2 = \nu$ and $\chi_n= (-1)^{n-1}\chi_{n-1}h_{n-1}\nu_n$ for $n\geq 3$. Observe that 
\[
p_n = (-1)^n\chi_n h_n
\]
and for $n\geq 3$
\[
m_n = p\chi_n i^{\otimes n} = (-1)^{n-1}p\chi_{n-1}h_{n-1}\nu_n i^{\otimes n} = p_{n-1} \nu_{n} i^{\otimes n}.
\]
Thus for all $n\geq 2$ we have
\[
m_n = p_{n-1} \nu_{n} i^{\otimes n}.
\]
Before proving the main result  of  this subsection we need one more technical lemma:
\begin{lemma} \label{lem:technical_2}
	For all $n\geq 1$ and for any decomposition $j_1+...+j_r=n$ we have
	\begin{align*}
		&(p_{j_1}\otimes ... \otimes p_{j_r})\nu_{n+1} h_{n+1}= (p_{j_1}\otimes ... \otimes p_{j_r})\\
		&\circ ((-1)^{j_2+...+j_{r}}\nu_{j_1+1}h_{j_1+1}\otimes ( ip )^{\otimes j_2+...+j_r}+...+\mathrm{id}^{\otimes j_1+...+j_{r-1}}\otimes \nu_{j_r+1}h_{j_{r}+1} ).
	\end{align*}
\end{lemma}
\begin{proof}
	For $n=1$ the Lemma is trivial. Fix $n\geq 1$ and assume the Lemma holds for all $1\leq m\leq n$. Let $j_1+...+j_{r+1} = n+1$, we compute
	\begin{align*}
		&(p_{j_1}\otimes ... \otimes p_{j_{r+1}})\nu_{n+1} h_{n+2}\\
		&= (p_{j_1}\otimes ... \otimes p_{j_r}\otimes p_{j_{r+1}})((-1)^{j_{r+1}}\nu_{j_1+...+j_r}\otimes \mathrm{id}^{\otimes j_{r+1}}+\mathrm{id}^{\otimes j_1+...+j_r}\otimes \nu_{j_{r+1}})h_{n+2}\\
		&= (-1)^{j_{r+1}}(p_{j_1}\otimes ... \otimes p_{j_r})\nu_{j_1+...+j_r}\otimes p_{j_{r+1}}(h_{j_1+...+j_r+1}\otimes (ip)^{\otimes j_{r+1}} + \mathrm{id}^{\otimes j_1+...+j_r+1}\otimes h_{j_{r+1}}) \\
		&p_{j_1}\otimes ... \otimes p_{j_r}\otimes p_{j_{r+1}}\nu_{j_{r+1}}(h_{j_1+...+j_r} \otimes (ip)^{\otimes j_{r+1}+1} + \mathrm{id}^{\otimes j_1+...+j_r}\otimes h_{j_{r+1}+1})\\
		&= (p_{j_1}\otimes ... \otimes p_{j_r})\nu_{j_1+...+j_r}h_{j_1+...+j_r+1} \otimes p_{j_{r+1}}(ip)^{\otimes j_{r+1}}\\
		&+ (-1)^{j_{r+1}}(p_{j_1}\otimes ... \otimes p_{j_r})\nu_{j_1+...+j_r} \otimes \underbrace{p_{j_{r+1}} h_{j_{r+1}}}_{=0}\\
		&\pm \underbrace{(p_{j_1}\otimes ... \otimes p_{j_r})h_{j_1+...+j_r}}_{=0}\otimes p_{j_{r+1}}\nu_{j_{r+1}}(ip)^{\otimes j_{r+1}+1}\\
		&+ p_{j_1}\otimes ... \otimes p_{j_r}\otimes p_{j_{r+1}}(\mathrm{id}^{\otimes j_1+...j_r} \otimes \nu_{j_{r+1}}) \mathrm{id}^{\otimes j_1+...+j_r}\otimes h_{j_{r+1}+1}\\
		&\overset{\mathrm{i.a.}}{=}\big[(p_{j_1}\otimes ... \otimes p_{j_r})((-1)^{j_2+...+j_r}\nu_{j_1}h_{j_1+1}\otimes ( ip )^{\otimes j_2+...+j_r}+...\\
		&\qquad\qquad\qquad\qquad\qquad...+ \mathrm{id}^{\otimes j_1+...+j_{r-1}}\otimes \nu_{j_r}h_{j_{r}+1} )\big] \otimes p_{j_{r+1}}(ip)^{\otimes j_{r+1}}\\
		&+p_{j_1}\otimes ... \otimes p_{j_{r+1}}( \mathrm{id}^{\otimes j_1+...+j_r} \otimes \nu_{j_{r+1}}h_{j_{r+1}+1})\\
		&= (-1)^{j_{r+1}}p_{j_1}\otimes ... \otimes p_{j_{r+1}} \big((-1)^{j_2+...+j_{r+1}}\nu_{j_1}h_{j_1+1}\otimes ( ip )^{\otimes j_2+...+j_{r}}+...\\
		&\qquad\qquad\qquad\qquad\qquad\qquad...+ \mathrm{id}^{\otimes j_1+...+j_{r-1}}\otimes \nu_{j_r}h_{j_{r}+1}\otimes (ip)^{\otimes j_{r+1}}\big) \\
		&+p_{j_1}\otimes ... \otimes p_{j_{r+1}}( \mathrm{id}^{\otimes j_1+...+j_r} \otimes \nu_{j_{r+1}}h_{j_{r+1}+1}).
	\end{align*}
\end{proof}
\begin{theorem} \label{thm:inverse_quasiiso}
	Let $(A,d,\nu)$ be a dg-algebra. Given a choice of deformation retract of the homology $H_*(A)$
	\[
	\begin{tikzcd}
		\arrow[loop left, "h"] (A,d) \arrow[r, shift left, "p"] & \arrow[l, shift left, "i"](H_*(A),0),
	\end{tikzcd}
	\]
	if $(H_*(A),m)$ is an $A_\infty$-structure, given by $m_1 =0$ and for $n\geq 2$ 
	\[
	m_n := -\alpha_{n} p\nu h_2\nu_3 ... h_{n-1} \nu_n i^{\otimes n}.
	\]
	Then the maps $p_1 = p$ and for $n\geq 2$ 
	\[
	p_n := (-1)^n p_{n-1} \nu_{n} h_n
	\]
	define an $A_\infty$-morphism $p:(A,d)\rightarrow (H_*(A),m)$.
\end{theorem}
\begin{proof}
	We must prove equation \eqref{eq:morphism} for all $n\geq 1$. For $n=1$ this is clear. Let $n\geq 2$, we must prove the following equation
	\begin{equation}\label{eq:p_is_a_morphism}
		p_nd_n + p_{n-1}\nu_n = \sum_{\substack{j_1+...+j_r=n\\r>1}}(-1)^lm_r(p_{j_1}\otimes ... \otimes p_{j_r}).
	\end{equation}
	The proof is by induction. For $n=2$, the left-hand side of equation \eqref{eq:p_is_a_morphism} is
	\begin{align*}
		p_2d_2 +p\nu &= p\nu h_2d_2 + p\nu \\
		&\overset{\eqref{eq:id-ip=d_nh_n}}{=} p\nu (-\mathrm{id}\otimes \mathrm{id}+ ip\otimes ip -d_2h_2) + p\nu\\
		&= p\nu(i\otimes i)(p\otimes p) -p\nu d_2 h_2\\
		&\overset{\eqref{eq:nu-d=d-nu}}{=} p\nu(i\otimes i)(p\otimes p)+ \underbrace{pd}_{=0} \nu h_2\\
		&= m_2(p\otimes p).
	\end{align*}
	Which is precisely the right-hand side of equation \eqref{eq:p_is_a_morphism}. Now assume equation \eqref{eq:p_is_a_morphism} holds for a fixed $n\geq 2$. Compute the left-hand side of equation \eqref{eq:p_is_a_morphism} for $n+1$:
	\begin{align*}
		p_{n+1}d_{n+1} + p_n\nu_{n+1} &= (-1)^{n+1} p_{n} \nu_{n+1} h_{n+1} d_{n+1} + p_n\nu_{n+1}\\
		&\overset{\eqref{eq:id-ip=d_nh_n}}{=} - p_{n} \nu_{n+1} (\mathrm{id}^{\otimes n+1} - (ip)^{\otimes n+1}- (-1)^{n}d_{n+1} h_{n+1}) + p_n\nu_{n+1} \\
		&\overset{\eqref{eq:nu-d=d-nu}}{=} p_{n} \nu_{n+1} (ip)^{\otimes n+1}- (-1)^{n}p_n d_n \nu_{n+1} h_{n+1}) \\
		&= m_{n+1}p^{\otimes n+1}-(-1)^n p_nd_n\nu_{n+1} h_{n+1} \\
		&\overset{\mathrm{i.a.}}{=} m_{n+1}p^{\otimes n+1} +(-1)^{n}p_{n-1}\nu_n \nu_{n+1} h_{n+1} \\
		&\ -(-1)^n\sum_{\substack{j_1+...+j_r=n\\r>1}}(-1)^l m_r (p_{j_1}\otimes ... \otimes p_{j_r})\nu_{n+1} h_{n+1}:= \clubsuit.
	\end{align*}
	We have $\nu_{n}\nu_{n+1} = 0$ due to associativity of the product -- this is easily proved by induction. By Lemma \ref{lem:technical_2} we have
	\begin{align*}
		&\ (p_{j_1}\otimes ... \otimes p_{j_r})\nu_{n+1} h_{n+1}\\
		&= \sum_{i=1}^r (-1)^{j_{i+1}+...+j_r} (p_{j_1}\otimes ... \otimes p_{j_r})( \mathrm{id}^{\otimes j_1+...+j_{i-1}}\otimes \nu_{j_i+1}h_{j_i+1}\otimes (ip)^{j_{i+1}+...+j_r})\\
		&= \sum_{i=1}^r(-1)^{r-i} (p_{j_1}\otimes ... \otimes p_{j_{i-1}} \otimes p_{j_i} \nu_{j_i+1}h_{j_i+1} \otimes p_{j_{i+1}}(ip)^{\otimes j_{i+1}}\otimes ... \otimes p_{j_r}(ip)^{\otimes j_{r}} \\
		&= \sum_{i=1}^r \delta_{j_{i+1},1}...\delta_{j_r,1}(-1)^{j_i+1+r-i} p_{j_1} \otimes ... \otimes p_{j_{i-1}} \otimes p_{j_i+1} \otimes p^{\otimes r-i}.
	\end{align*}
	Inserting this back into the previous equation we obtain
	\begin{align*}
		\clubsuit & = m_{n+1}p^{\otimes n+1} +  \sum_{r=2}^{n} \sum_{i=1}^r \sum_{\substack{j_1+...+j_i +(r-i) =n\\r>1}}\underbrace{(-1)^{l(j_1,...,j_i,1,...,1)} (-1)^{j_1+...+j_{i-1}}}_{=(-1)^{l(j_1,...,j_{i-1},j_i+1,1,...,1)}} \\
        & \qquad\qquad\qquad\qquad\qquad m_r ( p_{j_1} \otimes ... \otimes p_{j_{i-1}} \otimes p_{j_i+1} \otimes p^{\otimes r-i})\\
		&= m_{n+1}p^{\otimes n+1} +  \sum_{r=2}^{n} \sum_{i=1}^r \sum_{\substack{j_1+...+j_r=n+1\\j_i\neq 1, j_{i+1}=...=j_r=1\\r>1}} (-1)^l m_r(p_{j_1}\otimes ... \otimes p_{j_r})\\
		&= m_{n+1}p^{\otimes n+1} + \sum_{\substack{j_1+...+j_r=n+1\\1<r<n+1}}(-1)^l m_r(p_{j_1}\otimes ... \otimes p_{j_r}).
	\end{align*}
	Indeed, by equation \eqref{eq:A_infty_signs} we have
	\begin{align*}
		l(j_1,...,j_i,1,...,1) + j_1+...+j_{i-1} &= \sum_{1\leq a<b\leq i} (j_b-1)j_a + j_1+...+j_{i-1}\\
		&= \sum_{1\leq a<b< i} (j_b-1)j_a +\sum_{1\leq a< i} (j_i+1-1)j_a\\
		&= l(j_1,...,j_i+1,1,...,1).
	\end{align*}
\end{proof}

Theorem \ref{thm:inverse_quasiiso} allows for an explicit proof of the uniqueness part of Kadeishvili's Theorem.
\begin{theorem}
	Let $(A,d)$ be a dg-algebra and
\[
\begin{tikzcd}
		\arrow[loop left, "h"] (A,d) \arrow[r, shift left, "p"] & \arrow[l, shift left, "i"](H_*(A),0)
\end{tikzcd}
\]
a choice of deformation retract. Then $m_1 := 0$ and $m_n:= p\circ \lambda_n \circ i^{\otimes n}$ for $n\geq 2$ defines an $A_\infty$-structure on $H_*(A)$. The maps $i_1:=i$ and $i_n:= \lambda_n \circ i^{\otimes n}$ for $n\geq 2$ define a minimal model
\[
i:(H_*(A),m)\rightarrow (A,d).
\]
Moreover the minimal model is unique up to isomorphism.
\end{theorem}
\begin{proof}
	The first part is proved in Theorem \ref{thm:merkulov}. It remains to prove uniqueness: Let $i': (H_*(A),m')\rightarrow (A,d)$ be another minimal model. By Theorem \ref{thm:inverse_quasiiso} there exists a quasi-isomorphism $P: (A,d)\rightarrow (H_*(A),m)$ that lifts the projection map $p$. Then $P\circ i' :(H_*(A),m')\rightarrow (H_*(A),m)$ is an isomorphism as $(P\circ i')_1 = p\circ i= \mathrm{id}$.
\end{proof}

Another application of Theorem \ref{thm:inverse_quasiiso} an explicit proof of Lemma \ref{lem:quasi-isomorphic-inverse}.

\subsection{The Projection Map II} \label{extention_A_infty}
This subsection generalises Theorem \ref{thm:inverse_quasiiso} to the case where $A$ is allowed to be an $A_\infty$-algebra. More precisely, given an $A_\infty$-algebra $(A,m)$, let $H_*(A)$ denote the homology of the chain complex $(A,m_1)$, and let 
\[
\begin{tikzcd}
		\arrow[loop left, "h"] (A,m_1) \arrow[r, shift left, "p"] & \arrow[l, shift left, "i"](H_*(A),0)
\end{tikzcd}
\]
be a choice of deformation retract. We repeat the coalgebra construction from Section \ref{sec:pertub_lemma}, only this time the perturbation $t$ to the differential $\tilde d$ contains all the higher $m_n$ maps. That is, consider a map
\[
T^c(sA)\rightarrow sA
\]
which on $(sA)^{\otimes n}$ for $n\geq 2$ is given by 
\[
t_n := sm_n(s^{-1})^{\otimes n}:(sA)^{\otimes n}\rightarrow sA.
\]
This map extends to a unique coderivation 
\[
t: T^c(sA)\rightarrow T^c(sA).
\]
With minor changes the coalgebra argument from Section \ref{sec:pertub_lemma} goes through as before and gives us a square zero differential on $T^c(sH_*(A))$ given by
\[
d_\infty = P\sum_{n\geq 0} (-tH)^n tI.
\]
Set for $n\geq 1 $ and $ 1\leq j\leq n$
\[
m_j^n := \sum_{r+j+t=n} (-1)^{rj+t} \mathrm{id}^{\otimes r}\otimes m_j\otimes \mathrm{id}^{\otimes t}.
\]
Set $\chi_2 := m_2$, and for $n\geq 3$ set
\begin{equation} \label{eq:chi_2}
\chi_n := m_n + \sum_{j=1}^{n-2}(-1)^{n-j}\chi_{n-j}h_{n-j}m_{j+1}^n.
\end{equation}
We get, similar as in the previous subsection, an explicit formula for the $A_\infty$-structure on $H_*(A):$
\[
m'_n = p\chi_n i^{\otimes n}.
\]
Of course, the difference is that the $\chi_n$ definition includes all the higher $m_n$ maps. The coalgebra morphism coming from Section \ref{sec:pertub_lemma} is given by
\[
P_t = P\sum_{n\geq 0} (-tH)^n : T^c(sA)\rightarrow T^c(sH_*(A)).
\]
This gives an $A_\infty$-morphism $p:(A,m)\rightarrow (H_*(A),m')$ given by $p_1 =p$ and for $n\geq 2$:
\begin{equation}\label{eq:p_n_via_chi}
	p_n = (-1)^{n}p \chi_n h_n
\end{equation}
or more directly
\[
p_n = (-1)^{n} \sum_{j=1}^{n-1}p_{j}m_{n-j+1}^n h_n.
\]
This sets us up to prove that $p$ is an $A_\infty$-morphism explicitly as we are in a similar setting as in the previous subsection. Before proceeding we need to prove the following two technical lemmas. The first of which is concerning the $m_j^n$ maps\footnote{Equation \eqref{eq:technical_lemma} is equivalent to the higher associativity conditions, the proof is not really of interest, so feel free to skip it.}:
\begin{lemma} \label{lem:technical_lemma}
	For all $n\geq 2$ and $2\leq j\leq n$
	\begin{equation} \label{eq:technical_lemma}
		m_j^n m_1^n +\sum_{i=1}^{j-1} m_i^{n-j+i}m_{j-i+1}^n=0.
	\end{equation}
\end{lemma}
\begin{proof}
	Firstly, we show the case $n\geq 2$ and $j = n$ separately: the $A_\infty$-relations tell us that 
	\[
	0= \sum_{r+i+t=n} (-1)^{ri+t} m_{r+t+1}(\mathrm{id}^{\otimes r}\otimes m_i\otimes\mathrm{id}^{\otimes t}) = \sum_{i=1}^n m_{n-i+1} m_i^n,
	\]
	and thus
	\[
	m_n m_1^n = - \sum_{i=2}^n m_{n-i+1} m_i^n = - \sum_{i=1}^{n-1} m_{i} m_{n-i+1}^n.
	\]
	Secondly, we prove the rest by induction. The induction start $n=2$, $j=2$ is already proven. Assume the lemma holds for a fixed $n\geq 2$ and all $2\leq j \leq n$. Let $2\leq j \leq n$ and compute the first term of equation \eqref{eq:technical_lemma}
	\begin{align*}
		&m^{n+1}_j m_1^{n+1} \\
		&= (-m_j^n\otimes \mathrm{id}+ (-1)^{j(n-j+1)} \mathrm{id}^{\otimes n-j+1}\otimes m_j)m_1^{n+1}\\
		&= -m_j^n\otimes \mathrm{id}(-m_1^n \otimes\mathrm{id} + (-1)^n \mathrm{id}^{\otimes n} \otimes m_1)+(-1)^{j(n-j+1)} \mathrm{id}^{\otimes n-j+1}\\
		&\qquad\otimes m_j((-1)^j m_1^{n-j+1} \otimes \mathrm{id}^{\otimes j} + (-1)^{n-j+1} \mathrm{id}^{\otimes n-j+1} \otimes m_1^j)\\
		&= m_j^n m_1^n \otimes \mathrm{id} -(-1)^n m_j^n \otimes m_1 + (-1)^{j(n-j+1)} m_1^{n-j+1} \otimes m_j \\
		&\qquad+ (-1)^{(n-j+1)(j+1)} \mathrm{id}^{\otimes n-j+1} \otimes m_j m_1^j\\
		&\overset{\mathrm{i.a.}}{=} -\sum_{i=1}^{j-1} m_i^{n-j+i}m_{j-i+1}^n \otimes \mathrm{id} -(-1)^n m_j^n \otimes m_1 + (-1)^{j(n-j+1)} m_1^{n-j+1} \otimes m_j \\
		&\qquad - (-1)^{(n-j+1)(j+1)}\sum_{i=1}^{j-1} \mathrm{id}^{\otimes n-j+1} \otimes m_i m_{j-i+1}^j.
	\end{align*}
	Now compute the sum term of equation \eqref{eq:technical_lemma}
	\begin{align*}	
		& -\sum_{i=1}^{j-1} m_i^{n+1-j+i}m_{j-i+1}^{n+1}\\
		&= -\sum_{i=1}^{j-1} (-m_i^{n-j+i} \otimes \mathrm{id} + (-1)^{(n-j+1)i} \mathrm{id}^{\otimes n-j+1}\otimes m_i) m_{j-i+1}^{n+1}\\
		&= \sum_{i=1}^{j-1} m_i^{n-j+i} \otimes \mathrm{id}\big( -m_{j-i+1}^n \otimes \mathrm{id} +(-1)^{(n-j+i)(j-i+1)} \mathrm{id}^{\otimes n-j+i} \otimes m_{j-i+1}\big)\\
		&\qquad- (-1)^{(n-j+1)i} \mathrm{id}^{n-j+1}\otimes m_i\big((-1)^{i} m_{j-i+1}^{n-i+1} \otimes \mathrm{id}^{\otimes i}\\
		&\qquad\qquad\qquad\qquad\qquad\qquad+ (-1)^{(n-j+1)(j-i+1)} \mathrm{id}^{\otimes n-j+1} \otimes m^j_{j-i+1}\big)\\
		&= \sum_{i=1}^{j-1}- m_i^{n-j+i} m_{j-i+1}^n \otimes \mathrm{id} + (-1)^{(n-j+i)(j-i+1)} m_i^{n-j+i} \otimes m_{j-i+1}\\
		&\qquad - (-1)^{(n-i+1)i} m_{j-i+1}^{n-i+1}\otimes m_i- (-1)^{(n-j+1)(j+1)} \mathrm{id}^{\otimes n-j+1} \otimes m_i m^j_{j-i+1}.
	\end{align*}
	Rewrite the summands in the middle of the last line, setting $i' :=j-i+1$
	\begin{align*}
		&\quad \sum_{i=1}^{j-1} (-1)^{(n-j+i)(j-i+1)} m_i^{n-j+i} \otimes m_{j-i+1} - \sum_{i=1}^{j-1} (-1)^{(n-i+1)i} m_{j-i+1}^{n-i+1}\otimes m_i\\
		&= \sum_{i'=2}^{j} (-1)^{(n-i'+1)i'} m_{j-i'+1}^{n-i'+1} \otimes m_{i'} - \sum_{i=1}^{j-1} (-1)^{(n-i+1)i} m_{j-i+1}^{n-i+1}\otimes m_i\\
		&= (-1)^{(n-j+1)j} m_1^{n-j+1} \otimes m_j - (-1)^n m_j^n \otimes m_1.
	\end{align*}
\end{proof}
The second technical lemma is proved similar to Lemma \ref{lem:technical_2}.
\begin{lemma}\label{lem:technical_2_infty} For all $n\geq 1$, all $1\leq j\leq n$, all $1\leq r\leq j$, and all decompositions $j_1+...+j_r = j$ we have
\begin{align*}
    &(p_{j_1}\otimes ... \otimes p_{j_r})m_{n-j+1}^n h_n = 
    (p_{j_1}\otimes ... \otimes p_{j_r})\sum_{k=1}^r (-1)^{(j_1+...+j_{k-1})(n-j+1)+j_{k+1}+...+j_r}\\
    &\qquad\qquad\qquad\qquad\qquad\qquad \mathrm{id}^{\otimes j_1+...+j_{k-1}}\otimes m_{n-j+1}^{n-j+j_k}h_{n-j+j_k}\otimes (ip)^{j_{k+1}+... j_{j_r}}.
\end{align*}
\end{lemma}

\begin{theorem}
	Let $(A,m)$ be an $A_\infty$-algebra. Denote by $H_*(A)$ the homology of the chain complex $(A,m_1)$. Given a choice of deformation retract of the homology
	\[
	\begin{tikzcd}
		\arrow[loop left, "h"] (A,m_1) \arrow[r, shift left, "p"] & \arrow[l, shift left, "i"](H_*(A),0),
	\end{tikzcd}
	\]
	let $\chi_n$ be defined as in equation \eqref{eq:chi_2}. Assume that $m'_n := p\chi_n i^{\otimes n}$ defines an $A_\infty$-structure on $H_*(A)$. Then $p_1 := p$ and for $n\geq 2$ 
	\[
	p_n := (-1)^{n} \sum_{j=1}^{n-1}p_{j}m_{n-j+1}^n h_n
	\]
	defines an $A_\infty$-algebras morphism $p:(A,m)\rightarrow (H_*(A),m')$.
\end{theorem}
\begin{proof}
	We prove equation \eqref{eq:morphism} by induction on $n$. For $n= 1$, this is clear as
	\[
	pm_1 = 0 = m_1' p.
	\]
	Fix $n\geq 2$ and assume equation \eqref{eq:morphism} holds for all $1\leq i\leq n-1$. Then compute
	\begin{align*}
		&\sum_{r+j+t=n}(-1)^{rj+t}p_{r+t+1}(\mathrm{id}^{\otimes r}\otimes m_j \otimes \mathrm{id}^{\otimes t})\\
		&= \sum_{j=1}^n p_{n-j+1}m_j^n\\
		&= p_n m_1^n + \sum_{j=2}^n p_{n-j+1}m_j^n\\
		&= (-1)^{n} \sum_{j= 1}^{n-1} p_j m_{n-j+1}^n h_nm_1^n + \sum_{j'=1}^{n-1}p_{j'} m_{n-j'+1}^n := \clubsuit .
	\end{align*}
	Notice that
	\[
	h_nm_1^n = (-1)^{n-1}\mathrm{id}-(-1)^{n-1} (ip)^{\otimes n} -m_1^n h_n,
	\]
	and hence
	\begin{align*}
		\clubsuit &= \sum_{j= 1}^{n-1} p_j m_{n-j+1}^n (ip)^{\otimes n} +(-1)^{n-1} \sum_{j= 1}^{n-1} p_j m_{n-j+1}^n m_1^n h_n := \diamondsuit .
	\end{align*}
	The first term becomes
	\begin{align*}
		\sum_{j= 1}^{n-1} p_j m_{n-j+1}^n (ip)^{\otimes n} &\overset{\eqref{eq:p_n_via_chi}}{=} p m_n(ip)^{\otimes n} + p \sum_{j= 2}^{n-1}(-1)^{j} \chi_{j}h_j m_{n-j+1}^n (ip)^{\otimes n}\\
		&\overset{j':=n-j}{=} p m_n(ip)^{\otimes n} +p \underbrace{\sum_{j'= 1}^{n-2}(-1)^{n-j'} \chi_{n-j'}h_{n-j'} m_{j'+1}^n}_{=\chi_n-m_n} (ip)^{\otimes n}\\
		&= p\chi_n i^{\otimes n} p^{\otimes n}\\
		&= m_n'p^{\otimes n}.
	\end{align*}
	Thus
	\[
	\diamondsuit = m_n'p^{\otimes n} +(-1)^{n-1} \sum_{j= 1}^{n-1} p_j m_{n-j+1}^n m_1^n h_n=: \spadesuit.
	\]
	The first term is already of the proper form, let us focus on the second term. Applying Lemma \ref{lem:technical_lemma} we get
	\begin{align*}
		(-1)^{n-1} \sum_{j= 1}^{n-1} p_j m_{n-j+1}^n m_1^n h_n &= (-1)^{n} \sum_{j= 1}^{n-1} \sum_{i=1}^{n-j} p_j m_i^{j-1} m^n_{n-j-i+2} h_n.
	\end{align*}
To apply the induction assumtion we change the summation index $i':= i+j$
\begin{align*}
		&= (-1)^{n} \sum_{j= 1}^{n-1} \sum_{i'=j+1}^{n} p_j m_{i'-j}^{i'-1} m^n_{n-i'+2} h_n = (-1)^{n} \sum_{i'=2}^{n} \sum_{j= 1}^{i'-1} p_j m_{i'-j}^{i'-1} m^n_{n-i'+2} h_n\\
		&= (-1)^{n} \sum_{i'=2}^{n} \sum_{j'= 1}^{i'-1} p_{(i'-1)-j'+1} m_{j'}^{i'-1} m^n_{n-i'+2} h_n\\
		&\overset{\mathrm{i.a.}}{=} (-1)^{n+1} \sum_{i'=2}^{n}\sum_{r=1}^{i'-1} \sum_{j_1+...+j_r = i'-1}(-1)^{l(j_1,...,j_r)} m_r'(p_{j_1}\otimes ... \otimes p_{j_r}) m^n_{n-i'+2} h_n\\
		&= (-1)^{n} \sum_{r=1}^{n-1}\sum_{i'=r+1}^{n} \sum_{j_1+...+j_r = i'-1}(-1)^{l(j_1,...,j_r)} m_r'(p_{j_1}\otimes ... \otimes p_{j_r}) m^n_{n-i'+2} h_n.
	\end{align*}
	By Lemma \ref{lem:technical_2_infty} this becomes
\begin{align*}
    &= (-1)^{n} \sum_{r=1}^{n-1} \sum_{i'=r+1}^{n} \sum_{j_1+...+j_r = i'-1} (-1)^{l(j_1,...,j_r)} m_r'(p_{j_1}\otimes ... \otimes p_{j_r}) \\
    &\qquad \sum_{k=1}^r (-1)^{(j_1+...+j_{k-1})(n-i'+2)+j_{k+1}+...+j_r}\\
    &\qquad (\mathrm{id}^{\otimes j_1+...+j_{k-1}}\otimes m_{n-i'+2}^{n-i'+1+j_k}h_{n-i'+1+j_k} \otimes (ip)^{\otimes j_{k+1}+...+j_r})
\end{align*}
\begin{align*}
    &= (-1)^{n} \sum_{r=1}^{n-1} \sum_{i'=r+1}^{n} \sum_{j_1+...+j_r = i'-1} (-1)^{l(j_1,...,j_r)} m_r'\sum_{k=1}^r p_{j_1}\otimes ... \otimes p_{j_{k-1}}\\
    & \qquad(-1)^{(j_{k+1}+...+j_r-r+k)(n-i')} (-1)^{(j_1+...+j_{k-1})(n-i')+j_{k+1}+...+j_r}\\
    &\qquad \otimes p_{j_k}m_{n-i'+2}^{n-i'+1+j_k}h_{n-i'+1+j_k} \otimes p_{j_{k+1}}(ip)^{\otimes j_{k+1}} \otimes ... \otimes p_{j_r}(ip)^{\otimes j_{r}}.
\end{align*}
Observe that $p_{j_{k+1}}(ip)^{\otimes j_{k+1}} \otimes ... \otimes p_{j_r} (ip)^{\otimes j_r}= p^{\otimes j_{k+1}+...+j_r}\delta_{j_{k+1},1}...\delta_{j_r,1}$, hence we get
\begin{align*}
    &= \sum_{r=1}^{n-1} m_r'\sum_{k=1}^r \sum_{i'=r+1}^{n} \sum_{j_1+...+j_k+(r-k) = i'-1} (-1)^{l(j_1,...,j_k,1,...,1)}p_{j_1}\otimes ... \otimes p_{j_{k-1}}\\
    &\qquad (-1)^{(j_1+...+j_{k-1})(n-i')+n+r-k} \otimes p_{j_k} m_{n-i'+2}^{n-i'+1+j_k}h_{n-i'+1+j_k} \otimes p^{\otimes r-k}\\
    &=(-1)^n \sum_{r=1}^{n-1} m_r'\sum_{k=1}^r \sum_{i'=r+1}^{n} \sum_{j_1=1}^{i'-1+(r-1)}...\sum_{j_{k-1}=1}^{i'-1+j_1+...+j_{k-2}-(r-k+1)} p_{j_1}\otimes ... \otimes p_{j_{k-1}} \\
    & \qquad (-1)^{l(j_1,...,j_{k-1},i'-1-j_1-...-j_{k-1}-(r-k),1,...,1)}(-1)^{(j_1+...+j_{k-1})(n-i')+r-k}\\
    & \qquad \otimes p_{i'-1-j_1-...-j_{k-1}-(r-k)} m_{n-i'+2}^{n-r+k-j_1-...-j_{k-1}}h_{n-r+k-j_1-...-j_{k-1}} \otimes p^{\otimes r-k}.
\end{align*}
Permute the sum over $i'$ to the back and rewrite the sign using the definition of $l(j_1,...,j_r)$
\begin{align*}
    &= \sum_{r=1}^{n-1} m_r'\sum_{k=1}^r \sum_{j_1=1}^{n-r}...\sum_{j_{k-1}=1}^{n-r-j_1-...-j_{k-2}+(k-2)} \sum_{i'=r+j_1+...+j_{k-1}}^{n} p_{j_1}\otimes ... \otimes p_{j_{k-1}} \otimes \\
    & \qquad (-1)^{l(j_1,...,j_{k-1},n-1-j_1-...-j_{k-1}-(r-k),1,...,1)}(-1)^{j_1+...+j_{k-1}+n+r-k}\\
    &\qquad p_{i'-1-j_1-...-j_{k-1}-(r-k)} m_{n-i'+2}^{n-r+k-j_1-...-j_{k-1}}h_{n-r+k-j_1-...-j_{k-1}} \otimes p^{\otimes r-k}.
\end{align*}
Rename the index $i:= i'-1-(r-k)-j_1-...-j_{k-1}$ to obtain
\begin{align*}
    &= \sum_{r=1}^{n-1} m_r'\sum_{k=1}^r \sum_{j_1=1}^{n-r}...\sum_{j_{k-1}=1}^{n-r-j_1-...-j_{k-2}+(k-2)} p_{j_1}\otimes ... \otimes p_{j_{k-1}}\otimes \\
    & \qquad (-1)^{l(j_1,...,j_{k-1},n-1-j_1-...-j_{k-1}-(r-k),1,...,1)}(-1)^{n-(r-k)-j_1-...-j_{k-1}} \\
    &\sum_{i=1}^{n-(r-k)-j_1-...-j_{k-1}-1} p_{i} m_{n-r+k-j_1-...-j_{k-1}-i+1}^{n-r+k-j_1-...-j_{k-1}}h_{n-r+k-j_1-...-j_{k-1}} \otimes p^{\otimes r-k}.
\end{align*}
Recognize the inductive definition of $p_{n-r+k-j_1-...-j_{k-1}}$ to obtain
\begin{align*}
    &= \sum_{r=1}^{n-1} m_r'\sum_{k=1}^r \sum_{j_1=1}^{n-r}...\sum_{j_{k-1}=1}^{n-r-j_1-...-j_{k-2}+(k-2)} (-1)^{l(j_1,...,j_{k-1},n-1-j_1-...-j_{k-1}-(r-k),1,...,1)} \\
    & \qquad\qquad\qquad\qquad\qquad p_{j_1}\otimes ... \otimes p_{j_{k-1}} \otimes p_{n-r+k-j_1-...-j_{k-1}} \otimes p^{\otimes r-k}\\
    &= \sum_{r=1}^{n-1} m_r' \sum_{k=1}^r \sum_{\substack{j_1+...+j_r=n\\j_k\neq 1, j_{k+1}=...=j_r=1}} (-1)^{l(j_1,...,j_{r})} p_{j_1}\otimes ... \otimes p_{j_r}\\
    &= \sum_{r=1}^{n-1} m_r' \sum_{j_1+...+j_r=n} (-1)^{l(j_1,...,j_{r})} p_{j_1}\otimes ... \otimes p_{j_r}.
\end{align*}
It follows that 
\[
\spadesuit = \sum_{r=1}^{n} \sum_{j_1+...+j_r=n} (-1)^{l(j_1,...,j_{r})} m_r'(p_{j_1}\otimes ... \otimes p_{j_r})
\]
which completes the proof that $p$ is a morphism of $A_\infty$-algebras.
\end{proof}

\appendix
\section{Composition of \texorpdfstring{$A_\infty$}{TEXT}-Algebra Morphisms}
\label{sec:appendix_b}
In the following the signs have been neglected for ease of notation. Let $f:(A,m)\rightarrow (A',m')$ and $g:(A',m')\rightarrow(A'',m'')$ be two morphisms between $A_\infty$-algebras. Recall that composition of theses two morphisms is defined by
\[
(g\circ f)_n := \sum_{i_1+...+i_r=n} g_r(f_{i_1}\otimes ... \otimes f_{i_r}).
\]
We prove that this is well defined, i.e.\ the composition is again a morphism of $A_\infty$-algebras. Before getting started on the proof, let us show the following technical lemma:
\begin{lemma} \label{lem:appendix_techincal}
Let $\{f_n\}_{n\in \mathbb{N}}$ be a sequence of maps. Let $n\geq 2$ for all $2\leq r\leq n$ and all decompositions into positive integers $r_1+r_2=r$ we have
\begin{equation} \label{eq:appendix_f}
	\sum_{k=r_1}^{n-r_2}\sum_{i_1+...+i_{r_1}=k}\sum_{j_1+...+j_{r_2}=n-k}f_{i_1}\otimes ...\otimes f_{i_{r_1}}\otimes f_{j_1}\otimes ...\otimes f_{j_{r_2}}= \sum_{i_1+...+i_r=n}f_{i_1}\otimes ... \otimes f_{i_r}.
\end{equation}
\end{lemma}
\begin{proof}
	The proof is by induction on $n$ and $r$: For $n=r=2$ both sides of equation \eqref{eq:appendix_f} are equal to $f_1\otimes f_1$. For the rest of the proof fix $n\geq 2$. If $r=2$, then both sided of equation \eqref{eq:appendix_f} are equal to 
	\[
	\sum_{k=1}^{n-1}f_k\otimes f_{n-k}.
	\]
	Assume equation \eqref{eq:appendix_f} holds for all $2\leq m\leq n$ and all $2\leq r_1+r_2\leq m$. Then we prove equation \eqref{eq:appendix_f} for $n+1$ and $3\leq r_1+r_2\leq n+1$: Without loss of generality assume that $r_1\geq 2$. Starting on the left-hand side of equation \eqref{eq:appendix_f}
	\begin{align*}
		&\quad \sum_{k=1}^{n+1-r_2}\sum_{i_1+...+i_{r_1}=k} \sum_{j_1+...+j_{r_2}=n+1-k} f_{i_1}\otimes ...\otimes f_{i_{r_1}}\otimes f_{j_1}\otimes ...\otimes f_{j_{r_2}}\\
		&= \sum_{k=1}^{n+1-r_2} \sum_{i_1=1}^{k-r_1+1} \sum_{i_2+...+i_{r_1}=k-i_1} \sum_{j_1+...+j_{r_2}=n+1-k} f_{i_1}\otimes ...\otimes f_{i_{r_1}}\otimes f_{j_1}\otimes ...\otimes f_{j_{r_2}}.
	\end{align*}
	Interchanging the first two sums:
	\begin{align*}
		&= \sum_{i_1=1}^{n+2-r} f_{i_1}\otimes \sum_{k=r_1+i_1-1}^{n+1-r_2} \sum_{i_2+...+i_{r_1}=k-i_1} \sum_{j_1+...+j_{r_2}=n+1-k} f_{i_2}\otimes ...\otimes f_{i_{r_1}}\otimes f_{j_1}\otimes ...\otimes f_{j_{r_2}}\\
		& \overset{\mathrm{i.a.}}{=} \sum_{i_1=1}^{n+1-(r-1)} f_{i_1}\otimes\sum_{i_2+...+i_r=n+1-i_1} f_{i_2}\otimes ... \otimes f_{i_r}\\
		&= \sum_{i_1+...+i_r=n+1} f_{i_1}\otimes ... \otimes f_{i_r}.
	\end{align*}
\end{proof}

Recall the shorthand notion for $1\leq l\leq n$:
\[
m_l^n = \sum_{r+l+t=n}\mathrm{id}^{\otimes r}\otimes m_l\otimes \mathrm{id}^{\otimes t}.
\]
For $n$ arbitrary start computing the main equation:
\begin{align*}
	&\quad \sum_{r+l+t=n} (g\circ f)_{r+t+1}(\mathrm{id}^{\otimes r}\otimes m_l\otimes \mathrm{id}^{\otimes t})\\
	&= \sum_{l=1}^n (g\circ f)_{n-l+1}m_l^n\\
	&= \sum_{l=1}^n\sum_{r=1}^{n-l+1} \sum_{i_1+...+i_r=n-l+1} g_r(f_{i_1}\otimes ... \otimes f_{i_r}) m_l^n\\
	&= \sum_{r=1}^n\sum_{l=1}^{n-r+1} \sum_{i_1+...+i_r=n-l+1} g_r(f_{i_1}\otimes ... \otimes f_{i_r}) m_l^n.
\end{align*}
We split the last term as follows:
\begin{align*}
	(f_{i_1}\otimes ... \otimes f_{i_r}) m_l^n &= (f_{i_1}m_l^{i_1+l-1} \otimes f_{i_2}... \otimes f_{i_r})+...+(f_{i_1}\otimes ... \otimes f_{i_{r-1}} \otimes f_{i_r} m_l^{i_r+l-1}).
\end{align*}
Insert this into the main equation:
\begin{align*}
	&\sum_{r=1}^n\sum_{l=1}^{n-r+1} \sum_{i_1+...+i_r=n-l+1} g_r(f_{i_1}m_l^{i_1+l-1} \otimes f_{i_2}... \otimes f_{i_r}))+...\\
	&...+ \sum_{r=1}^n\sum_{l=1}^{n-r+1} \sum_{i_1+...+i_r=n-l+1} g_r(f_{i_1}\otimes ... \otimes f_{i_{r-1}} \otimes f_{i_r} m_l^{i_r+l-1})
\end{align*}
and push the sum over $l$ to the back:
\begin{align*}
	&= \sum_{r=1}^n \big[\sum_{i_2=1}^{n-r+1} ... \sum_{i_r=1}^{n-1-i_2-...-i_{r-1}} \sum_{l=1}^{n-i_2-...-i_{r}} g_r(f_{n-l+1-i_2-...-i_r}m_l^{n-i_2-...-i_r} \otimes f_{i_2}... \otimes f_{i_r}))+... \\
	&+ \sum_{i_1=1}^{n-r+1}... \sum_{i_{r-1}=1}^{n-1-i_1-...-i_{r-2}}\sum_{l=1}^{n-i_1-...-i_{r-1}}g_r(f_{i_1}\otimes ... \otimes f_{i_{r-1}} \otimes f_{n-l+1-i_1-...-i_{r-1}} m_l^{n-i_1-...-i_{r-1}})\big].
\end{align*}
Now using the fact that $f$ is a morphism:
\begin{align*}
	&= \sum_{r=1}^n \big[\sum_{i_2=1}^{n-r+1} ... \sum_{i_r=1}^{n-1-i_2-...-i_{r-1}} \sum_{j_1+...+j_s=n-i_2-...-i_r} g_r(m'_s (f_{j_1}\otimes ... \otimes f_{j_s}) \otimes f_{i_2}... \otimes f_{i_r}))+... \\
	&+ \sum_{i_1=1}^{n-r+1}... \sum_{i_{r-1}=1}^{n-1-i_1-...-i_{r-2}}\sum_{j_1+...+j_s=n-i_1-...-i_{r-1}} g_r(f_{i_1}\otimes ... \otimes f_{i_{r-1}} \otimes m'_s (f_{j_1}\otimes ... \otimes f_{j_s}) )\big].
\end{align*}
Note that there is a sum over $s$ hiding in the notation. Push the sum over $s$ to the front:
\begin{align*}
	&= \sum_{s=1}^n \sum_{r=1}^{n-s+1}\big[\sum_{i_2=1}^{n+2-r-s} ... \sum_{i_r=1}^{n-i_2-...-i_{r-1}-s} \sum_{j_1+...+j_s=n-i_2-...-i_r} g_r(m'_s (f_{j_1}\otimes ... \otimes f_{j_s}) \otimes f_{i_2}... \otimes f_{i_r}))+... \\
	&+ \sum_{i_1=1}^{n+2-r-s}... \sum_{i_{r-1}=1}^{n-i_1-...-i_{r-2}-s}\sum_{j_1+...+j_s=n-i_1-...-i_{r-1}} g_r(f_{i_1}\otimes ... \otimes f_{i_{r-1}} \otimes m'_s (f_{j_1}\otimes ... \otimes f_{j_s}) )\big].
\end{align*}
Now move the $g$ and $m'$ terms in front of the sums:
\begin{align*}
	&= \sum_{s=1}^n \sum_{r=1}^{n-s+1}g_r\big[m'_s \otimes \mathrm{id}^{\otimes r-1}\sum_{i_2=1}^{n+2-r-s} ... \sum_{i_r=1}^{n-i_2-...-i_{r-1}-s} \sum_{j_1+...+j_s=n-i_2-...-i_r} \\
	&\qquad\qquad\qquad\qquad\qquad\qquad\qquad\qquad\qquad\qquad (f_{j_1}\otimes ... \otimes f_{j_s} \otimes f_{i_2}... \otimes f_{i_r}))+... \\
	&+ \mathrm{id}^{\otimes r-1}\otimes m_s'\sum_{i_1=1}^{n+2-r-s}... \sum_{i_{r-1}=1}^{n-i_1-...-i_{r-2}-s}\sum_{j_1+...+j_s=n-i_1-...-i_{r-1}}\\
	&\qquad\qquad\qquad\qquad\qquad\qquad\qquad\qquad\qquad\qquad (f_{i_1}\otimes ... \otimes f_{i_{r-1}} \otimes f_{j_1}\otimes ... \otimes f_{j_s} )\big].
\end{align*}
Notice that the index $s$ -- appearing in the sum sums over $i$ and $j$ -- is a dummy index and after renaming the other indices we get
\begin{align*}
	&= \sum_{s=1}^n \sum_{r=1}^{n-s+1}g_r \big[m'_s \otimes \mathrm{id}^{\otimes r-1} \sum_{i_1+...+i_{r+s-1}=n} (f_{i_1}\otimes ... \otimes f_{i_{r+s-1}}) +... \\
	& \qquad\qquad\qquad + \mathrm{id}^{\otimes r-1}\otimes m_s' \sum_{i_1+...+i_{r+s-1}=n} (f_{i_1}\otimes ... \otimes f_{i_{r+s-1}} )\big].
\end{align*}
Now set $r':= r+s-1$
\[
\sum_{s=1}^n \sum_{r'=s}^{n}g_{r'-s+1} \big[m'_s \otimes \mathrm{id}^{\otimes r'-s} +... + \mathrm{id}^{\otimes r'-s}\otimes m_s' \big] \sum_{i_1+...+i_{r'}=n} (f_{i_1}\otimes ... \otimes f_{i_{r'}} ).
\]
Rename the index $r'$ to $r$ and then interchange the first two sums:
\begin{align*}
	&\quad \sum_{s=1}^n \sum_{r=s}^{n}g_{r-s+1} {m'}_s^{r} \sum_{i_1+...+i_r=n} (f_{i_1}\otimes ... \otimes f_{i_r} )\\
	&= \sum_{r=1}^n \sum_{s=1}^r g_{r-s+1} {m'}_s^{r} \sum_{i_1+...+i_r=n} (f_{i_1}\otimes ... \otimes f_{i_r} )\\
	&= \sum_{r=1}^n \sum_{s=1}^{r} \sum_{j_1+...+j_s=r} m_s'' (g_{j_1}\otimes...\otimes g_{j_s}) \sum_{i_1+...+i_r=n} (f_{i_1}\otimes ... \otimes f_{i_r} )\\
	&= \sum_{s=1}^n m_s'' \sum_{r=s}^{n} \sum_{j_1+...+j_s=r} \sum_{i_1+...+i_r=n} (g_{j_1}\otimes...\otimes g_{j_s}) (f_{i_1}\otimes ... \otimes f_{i_r})
\end{align*}
It remains to show that
\begin{align} \label{eq:appendix_induction}
	&\sum_{r=s}^{n} \sum_{j_1+...+j_s=r} \sum_{i_1+...+i_r=n} (g_{j_1}\otimes...\otimes g_{j_s}) (f_{i_1}\otimes ... \otimes f_{i_r})\nonumber \\
	&= \sum_{i_1+...+i_s=n}(g\circ f)_{i_1}\otimes ...\otimes (g\circ f)_{i_s}
\end{align}
for all $n$ and all $1\leq s\leq n$. The proof is by induction: The induction start is $n=s=1$ where both sides of equation \eqref{eq:appendix_induction} are equal to $g_1\circ f_1$. Fix $n\geq 1$ and assume \eqref{eq:appendix_induction} is true for all $1\leq m\leq n$ and all $1\leq s\leq m$. First observe that for $n+1$ and $s=1$ both sides of equation \eqref{eq:appendix_induction} are equal to $(g\circ f)_n$. Now let $2\leq s\leq n+1$ and start with the right-hand side of equation \eqref{eq:appendix_induction}:
\begin{align*}
	&\sum_{i_1+...+i_s=n+1}(g\circ f)_{i_1}\otimes ...\otimes (g\circ f)_{i_s}\\
	&= \sum_{i_1=1}^{n+2-s}(g\circ f)_{i_1}\otimes \sum_{i_2+...+i_s=n+1-i_1}(g\circ f)_{i_2}\otimes ...\otimes (g\circ f)_{i_s}\\
	&\overset{\mathrm{i.a.}}{=} \sum_{i_1=1}^{n+2-s}(g\circ f)_{i_1}\otimes \sum_{r=s-1}^{n+1-i_1} \sum_{j_2+...+j_s =r} \sum_{k_1+...+k_r=n+1-i_1} (g_{j_2}\otimes...\otimes g_{j_s}) (f_{k_1}\otimes ... \otimes f_{k_r})\\
	&\overset{k:=i_1}{=} \sum_{k=1}^{n+2-s}\sum_{r=s-1}^{n+1-k}\sum_{l=1}^k \sum_{i_1+...+i_l=k} \sum_{j_2+...+j_s =r} \sum_{k_1+...+k_r=n+1-k}\\
	&\qquad \qquad \qquad \qquad (g_l\otimes g_{j_2}\otimes ...\otimes g_{j_s}) (f_{1_1}\otimes ...\otimes f_{i_l}\otimes f_{k_1}\otimes ... \otimes f_{k_r}).
\end{align*}
Now set $r':= n+1-r$ 
\begin{align*}
	&\quad \sum_{k=1}^{n+2-s}\sum_{r'=k}^{n+2-s}\sum_{l=1}^k \sum_{i_1+...+i_l=k} \sum_{j_2+...+j_s =n+1-r'} \sum_{k_1+...+k_{n+1-r'}=n+1-k}\\
	&\qquad \qquad \qquad \qquad (g_l\otimes g_{j_2}\otimes ...\otimes g_{j_s}) (f_{1_1}\otimes ...\otimes f_{i_l}\otimes f_{k_1}\otimes ... \otimes f_{k_{n+1-r'}}).
\end{align*}
Now interchange the sums as follows
\[
\sum_{k=1}^{n+2-s}\sum_{r'=k}^{n+2-s}\sum_{l=1}^k = \sum_{r'=1}^{n+2-s}\sum_{k=1}^{r'}\sum_{l=1}^k= \sum_{r'=1}^{n+2-s}\sum_{l=1}^{r'}\sum_{k=l}^{r'}.
\]
The main equation becomes
\begin{align*}
	&\sum_{r'=1}^{n+2-s}\sum_{l=1}^{r'}\sum_{k=l}^{r'} \sum_{i_1+...+i_l=k} \sum_{j_2+...+j_s =n+1-r'} \sum_{k_1+...+k_{n+1-r'}=n+1-k}\\
	&\qquad \qquad \qquad \qquad (g_l\otimes g_{j_2}\otimes ...\otimes g_{j_s}) (f_{1_1}\otimes ...\otimes f_{i_l}\otimes f_{k_1}\otimes ... \otimes f_{k_{n+1-r'}})\\
	&= \sum_{r'=1}^{n+2-s}\sum_{l=1}^{r'} \sum_{j_2+...+j_s =n+1-r'} (g_l\otimes g_{j_2}\otimes ...\otimes g_{j_s}) \\
	&\quad \sum_{k=l}^{r'} \sum_{i_1+...+i_l=k} \sum_{k_1+...+k_{n+1-r'}=n+1-k} (f_{1_1}\otimes ...\otimes f_{i_l}\otimes f_{k_1}\otimes ... \otimes f_{k_{n+1-r'}})\\
	&= \sum_{r'=1}^{n+2-s}\sum_{l=1}^{r'} \sum_{j_2+...+j_s =n+1-r'} (g_l\otimes g_{j_2}\otimes ...\otimes g_{j_s}) \sum_{i_1+...+i_{n+1-r'+l}=n+1}(f_{i_1}\otimes ... \otimes f_{i_{n+1-r'+l}})
\end{align*}
Where the last equality uses Lemma \ref{lem:appendix_techincal}. Shift the index $l$ in the following way $r:= r'-l+1$
\begin{align*}
	&= \sum_{r'=1}^{n+2-s}\sum_{r=1}^{r'} \sum_{j_2+...+j_s =n+1-r'} (g_{r'-r+1}\otimes g_{j_2}\otimes ...\otimes g_{j_s}) \sum_{i_1+...+i_{n+2-r}=n+1} (f_{i_1}\otimes ... \otimes f_{i_{n+2-r}})
\end{align*}
Then after interchanging the sums
\[
\sum_{r'=1}^{n+2-s}\sum_{r=1}^{r'} = \sum_{r=1}^{n+2-s}\sum_{r'=r}^{n+2-s}
\]
shift the index $r'$ as follows $j_1 := r'-r+1$ 
\begin{align*}
	& \sum_{r=1}^{n+2-s}\sum_{j_1=1}^{n+3-s-r} \sum_{j_2+...+j_s =n+2-r-j_1} (g_{j_1}\otimes ...\otimes g_{j_s})\sum_{i_1+...+i_{n+2-r}=n+1} (f_{i_1}\otimes ... \otimes f_{i_{n+2-r}})\\
	&= \sum_{r=1}^{n+2-s} \sum_{j_1+...+j_s =n+2-r} (g_{j_1}\otimes ...\otimes g_{j_s}) \sum_{i_1+...+i_{n+2-r}=n+1} (f_{i_1}\otimes ... \otimes f_{i_{n+2-r}}).
\end{align*}
Finally shift the index $r$ as follows $r':= n+2-r$ we obtain
\[
\sum_{r=s}^{n+1} \sum_{j_1+...+j_s =r'} (g_{j_1}\otimes ...\otimes g_{j_s}) \sum_{i_1+...+i_{r'}=n+1} (f_{i_1}\otimes ... \otimes f_{i_{r'}}).
\]
This completes the induction step. 
\newpage
\printbibliography	
\end{document}